\documentclass[11pt, reqno]{amsart} \textwidth=14.9cm \oddsidemargin=1cm \evensidemargin=1cm
\usepackage{amsmath,amstext,amsbsy,amssymb,amscd}
\usepackage{amsmath}
\usepackage{amsxtra}
\usepackage{amscd}
\usepackage{amsthm}
\usepackage{amsfonts}
\usepackage{amssymb}
\usepackage{eucal}
\usepackage{color}
\usepackage{graphicx}%
\usepackage{color}
\newtheorem{theorem}{Theorem}[section]
\newtheorem{lemma}[theorem]{Lemma}
\newtheorem{proposition}[theorem]{Proposition}
\newtheorem{corollary}[theorem]{Corollary}
\theoremstyle{definition}
\newtheorem{definition}[theorem]{Definition}
\newtheorem{remark}[theorem]{Remark}

\definecolor{A}{rgb}{.75,1,.75}

\newcommand{\Z}{\mathbb{Z}}

\numberwithin{equation}{section}

\begin{document}

\title[affine and cyclotomic Yokonuma-Hecke algebras]{Modular representation theory of affine and cyclotomic Yokonuma-Hecke algebras}
\author[Weideng Cui and Jinkui Wan]{Weideng Cui and Jinkui Wan}
\address{(Cui) School of Mathematics, Shandong University, Jinan, Shandong 250100, P.R. China.}
\email{cwdeng@amss.ac.cn}

\address{
(Wan) School of Mathematics, Beijing Institute of Technology,
Beijing, 100081, P.R. China. } \email{wjk302@hotmail.com}

\begin{abstract}
We explore the modular representation theory of affine and cyclotomic Yokonuma-Hecke algebras.  We provide an equivalence between the category of finite dimensional representations of the affine (resp. cyclotomic) Yokonuma-Hecke algebra and that of an algebra which is a direct sum of tensor products of affine Hecke algebras of type $A$ (resp. Ariki-Koike algebras). As one of the applications, the irreducible representations of affine and cyclotomic Yokonuma-Hecke algebras are classified over an algebraically closed field of characteristic $p$. Secondly, the modular branching rules for these algebras are obtained; moreover, the resulting modular branching graphs for cyclotomic Yokonuma-Hecke algebras are identified with crystal graphs of irreducible integrable representations of affine Lie algebras of type $A.$
\end{abstract}



\thanks{\emph{\emph{2010} Mathematics Subject Classification}. 20C08.}
\thanks{\emph{Corresponding author:} Jinkui Wan, wjk302@hotmail.com}

\maketitle
\medskip
\section{Introduction}
\subsection{}
For the symmetric group $\mathfrak{S}_{n},$ Kleshchev \cite{Kle1} discovered the $p$-modular branching rules for irreducible representations of $\mathfrak{S}_{n}.$ Later on, Lascoux, Leclerc and Thibon \cite{LLT} established a close connection between global crystal bases of basic $U_q(\widehat{\mathfrak{sl}}_n)$-modules and modular representations of $\mathfrak{S}_{n},$ or more generally representations of Hecke algebras of type $A$ at roots of unity. The observation \cite{LLT} turned out to be a beginning of an exciting development which continues to this day,
including a development of deep connections between (affine, cyclotomic or degenerate affine) Hecke algebras of type $A$ at the $\ell$-th roots of unity, or cyclotomic quiver Hecke algebras of type $A$ associated with dominant integral weights of level $\ell,$ and canonical bases for integrable $U_q(\widehat{\mathfrak{sl}}_\ell)$-modules via categorification; see, e.g., [Ari1, BK1-3, BKW, Br, Gr, GV, Kle2] for related works.



\subsection{}
Yokonuma-Hecke algebras of general types were first introduced in the sixties by Yokonuma \cite{Yo}. In the late 1990s and early 2000s, Juyumaya and Kannan \cite{Ju1, JuK} gave a new presentation of the Yokonuma-Hecke algebra $Y_{r,n}(q)$ of type $A$, and since then it has been commonly used for studying this algebra. By Juyumaya and Kannan's presentation, the Yokonuma-Hecke algebra $Y_{r,n}(q)$ can be regarded as a deformation of the group algebra of the wreath product $(\mathbb{Z}/r\mathbb{Z})\wr \mathfrak{S}_{n}$ of the cyclic group $\mathbb{Z}/r\mathbb{Z}$ and the symmetric group $\mathfrak{S}_n$. It is well known that there exists another deformation of the group algebra of the wreath product $(\mathbb{Z}/r\mathbb{Z})\wr \mathfrak{S}_{n},$ namely the Ariki-Koike algebra \cite{AK}. The Yokonuma-Hecke algebra $Y_{r,n}(q)$ differs from the Ariki-Koike algebra in the way that $Y_{r,n}(q)$ is isomorphic to the modified Ariki-Koike algebra defined by Shoji \cite{S} according to Espinoza and Ryom-Hansen's work \cite{ER}.

Recently, by generalizing the approach of Okounkov-Vershik \cite{OV} to the representation theory of $\mathfrak{S}_n$, Chlouveraki and Poulain d'Andecy \cite{ChP1} introduced the notion of affine Yokonuma-Hecke algebra $\widehat{Y}_{r,n}(q)$ and gave an explicit construction for all irreducible representations of $Y_{r,n}(q)$ over $\mathbb{C}(q)$, and further obtained a semisimplicity criterion for $Y_{r,n}(q)$. In their subsequent paper \cite{ChP2}, they studied the representation theory of the affine Yokonuma-Hecke algebra $\widehat{Y}_{r,n}(q)$ and the cyclotomic Yokonuma-Hecke algebra $Y_{r,n}^{\lambda}(q)$. In particular, they gave the classification of irreducible representations of $Y_{r,n}^{\lambda}(q)$ in the generic semisimple case. In the past several years, the study of affine and cyclotomic Yokonuma-Hecke algebras has made substantial progress; see, e.g., \cite{ChJuKL, ChP1, ChP2, ChPo, ChS, C, ER, JPA, Lu3, PA, Ro, RS}.

\subsection{}

The affine Yokonuma-Hecke algebra $\widehat{Y}_{r,n}(q)$ introduced in \cite{ChP1} can actually be defined over an algebraically closed field $\mathbb{K}$ of characteristic $p$ such that $p$ does not divide $r$. Throughout the paper we shall denote by $\widehat{Y}_{r,n}$ (see Definition \ref{rel-def-Y1}) the affine Yokonuma-Hecke algebra over $\mathbb{K},$ and denote by $Y_{r,n}^{\lambda}$ (see \eqref{cyclotomic-yhc-alge-defini} for the definition) the associated cyclotomic Yokonuma-Hecke algebra. It turns out that the affine Yokonuma-Hecke algebra $\widehat{Y}_{r,n}$ has a degenerate version which is the so-called wreath Hecke algebra introduced by the second author and Wang in \cite{WW} which also includes an exploration of the modular representation theory and modular branching rules for wreath Hecke algebras. This paper is aimed to study the modular representation theory of the affine Yokonuma-Hecke algebras $\widehat{Y}_{r,n}$ by generalizing the approach of \cite{WW}. A classification of simple $\widehat{Y}_{r,n}$-modules as well as the classification of simple $Y_{r,n}^{\lambda}$-modules is provided. Meanwhile, we obtain the modular branching rules for $\widehat{Y}_{r,n}$ and $Y_{r,n}^{\lambda}$ respectively, and establish a connection to the crystal graphs of simple integrable modules of affine Lie algebras of type $A.$

\subsection{}
In this subsection we briefly introduce the framework of this article. In Section 2 we give an explicit description of the center of $\widehat{Y}_{r,n}.$

As an analog of the category equivalence established in \cite[Section 3]{WW}, an equivalence between the category of finite dimensional $\widehat{Y}_{r,n}$-modules and the module category of an algebra which is a direct sum of tensor products of various affine Hecke algebras of type $A$ is achieved in Section 3.

In Section 4, we will give three applications of the above module category equivalence. First of all, we provide the classification of simple $\widehat{Y}_{r,n}$-modules by using the known classification of simple modules for various affine Hecke algebras of type $A.$ As a second application, we establish the modular branching rule for $\widehat{Y}_{r,n}$. That is, we describe explicitly the socle of the restriction of a simple $\widehat{Y}_{r,n}$-module to a subalgebra $\widehat{Y}_{r,(n-1,1)}$ (see \eqref{eq:Young-sub}), and hence to the subalgebra $\widehat{Y}_{r,n-1}.$ Finally, we give a block decomposition in the category of finite dimensional $\widehat{Y}_{r,n}$-modules.

We then extend the equivalence established in Section 3 to the category of finite dimensional $Y_{r,n}^{\lambda}$-modules in Section 5 and present several applications in Section 6. Firstly, we give the classification of simple $Y_{r,n}^{\lambda}$-modules by applying the known classification of simple modules for various Ariki-Koike algebras. Secondly, we define an action of the affine Lie algebra $\widehat{sl}_{e}^{\oplus r},$ which is a direct sum of $r$ copies of $\widehat{sl}_{e}$ ($e$ denotes the order of $q$ in $\mathbb{K}^*$), on the direct sum of the Grothendieck groups of the module categories of $Y_{r,n}^{\lambda}$ over all $n\geq 0,$ and further show that the resulting representation is irreducible. Thirdly, we establish the modular branching rules for $Y_{r,n}^{\lambda}.$ That is, we describe explicitly the socle of the restriction of a simple $Y_{r,n}^{\lambda}$-module to a subalgebra $Y_{r,(n-1,1)}^{\lambda},$  and hence to the subalgebra $Y_{r,n-1}^{\lambda};$ moreover, we show that the modular branching graph for $Y_{r,n}^{\lambda}$ is isomorphic to the corresponding crystal graph of the simple $\widehat{sl}_{e}^{\oplus r}$-module $L(\lambda)^{\otimes r}.$ Finally, we give the classification of blocks for $Y_{r,n}^{\lambda},$ which is reduced to the known classification of blocks for the Ariki-Koike algebra due to Lyle and Mathas \cite{LM}.

Throughout the paper we assume that $r, n\in \mathbb{Z}_{\geq 1}$ and $\mathbb{K}$ is an algebraically closed field of characteristic $p$ such that $p$ does not divide $r$ (note that $p=0$ is possible). We remark that the assumption that $p$
does not divide $r$ is required so that the affine Yokonuma-Hecke algebras $\widehat{Y}_{r,n}$ are well-defined
over $\mathbb{K}$. We fix an invertible element $q\in \mathbb{K}$ and further assume that $q\neq 1.$



{\it Additional remark}: The first version of this paper was made available on arXiv in June 2015 (arXiv:1506.06570). Later on in 2016, Poulain d'Andecy posted his preprint \cite{PA} on arXiv, in which he established the algebra isomorphism between the affine (resp. cyclotomic) Yokonuma-Hecke algebra and a direct sum of matrix algebras with entries in affine Hecke algebras of type $A$ (resp. Ariki-Koike algebras); the isomorphism theorem for cyclotomic Yokonuma-Hecke algebras has been subsequently reobtained by Rostam \cite{Ro}. From their results, one can recover the category equivalences established in this paper using a different approach. It is also worthwhile to point out that the approach used in our paper or in \cite{WW} has been recently applied by Savage in \cite{Sa} to introduce and study the so-called affine wreath product algebras which appear naturally in Heisenberg categorification and in particular include the various known algebras such as degenerate affine Hecke algebras and wreath Hecke algebras as special cases.

\section{Definition and properties of affine Yokonuma-Hecke algebras}
In this section we first recall the definition of the affine Yokonuma-Hecke algebra $\widehat{Y}_{r,n}$ and introduce some necessary results following [ChPA1-2]. Then we describe the center of $\widehat{Y}_{r,n}.$
\subsection{The definition of affine Yokonuma-Hecke algebras}
\begin{definition}{\rm (See \cite[$\S3.1$]{ChP1}.)} The \emph{affine Yokonuma-Hecke algebra}, denoted by $\widehat{Y}_{r,n}=\widehat{Y}_{r,n}(q)$, is a $\mathbb{K}$-associative algebra generated by the elements $t_{1},\ldots,t_{n},$ $g_{1},\ldots,g_{n-1},$ $X_{1}^{\pm1}$ with relations:
\begin{equation}\label{rel-def-Y1}\begin{array}{rclcl}
g_{i}^{2}\hspace*{-7pt}&=&\hspace*{-7pt}q+(q-1)g_{i}e_{i} && \mbox{for $1\leq i\leq n-1$,}\\[0.1em]
g_ig_j\hspace*{-7pt}&=&\hspace*{-7pt}g_jg_i && \mbox{for $1\leq i,j\leq n-1$ such that $\vert i-j\vert \geq 2$,}\\[0.1em]
g_ig_{i+1}g_i\hspace*{-7pt}&=&\hspace*{-7pt}g_{i+1}g_ig_{i+1} && \mbox{for $1\leq i\leq n-2$,}\\[0.1em]
t_it_j\hspace*{-7pt}&=&\hspace*{-7pt}t_jt_i &&  \mbox{for $1\leq i,j\leq n$,}\\[0.1em]
t_i^r\hspace*{-7pt}&=&\hspace*{-7pt}1 && \mbox{for $1\leq i\leq n$,}\\[0.1em]
g_it_j\hspace*{-7pt}&=&\hspace*{-7pt}t_{s_i(j)}g_i && \mbox{for $1\leq i\leq n-1$ and $1\leq j\leq n$,}\\[0.1em]
\end{array}
\end{equation}
and with the following relations involving $X_{1}^{\pm 1}$:
\begin{equation}\label{rel-def-Y333-daadaaa}\begin{array}{rclcl}
X_{1}X_{1}^{-1}\hspace*{-7pt}&=&\hspace*{-7pt}X_{1}^{-1}X_{1}=1,\\[0.1em]
g_{1}X_{1}g_{1}X_{1}\hspace*{-7pt}&=&\hspace*{-7pt}X_{1}g_{1}X_{1}g_{1}, \\[0.1em]
g_{i}X_{1}\hspace*{-7pt}&=&\hspace*{-7pt}X_{1}g_{i} &&  \mbox{for $2\leq i\leq n-1$,}\\[0.1em]
t_{j}X_{1}\hspace*{-7pt}&=&\hspace*{-7pt}X_{1}t_{j} && \mbox{for $1\leq j\leq n$,}
\end{array}
\end{equation}
where $s_{i}$ is the transposition $(i,i+1)$ in the symmetric group $\mathfrak{S}_n$ on $n$ letters, and for each $1\leq i\leq n-1$,
\begin{equation*}\label{rel-def-Y2-addadd}
e_{i} :=\frac{1}{r}\sum\limits_{s=0}^{r-1}t_{i}^{s}t_{i+1}^{-s}.
\end{equation*}
\end{definition}

We assume that $\widehat{Y}_{r,0}=\mathbb{K}.$ Note that the quadratic relations in \eqref{rel-def-Y1} are different from those in \cite[$(3.1)$]{ChP1}. The formalization here follows from \cite[$(3.2)$]{ChPo}.

\begin{remark}\label{rem-YH}
We recall that the \emph{Yokonuma-Hecke algebra} $Y_{r,n}=Y_{r,n}(q)$ of type A, first defined by Yokonuma in \cite{Yo}, is an associative
algebra over $\mathbb{K}$ generated by elements $t'_{1},\ldots,t'_{n}$ and $g'_{1},\ldots,g'_{n-1}$ with the defining relations as in (\ref{rel-def-Y1}) with each $g_i$ replaced by $g'_i$ and each $t_j$ replaced by $t'_j$ \cite{Ju1, Ju2, JuK}. By \cite[(2.6)]{ChP2}, the homomorphism $\iota: Y_{r,n}\rightarrow \widehat{Y}_{r,n},$ which is defined by
\begin{equation*}
\iota(t'_j)=t_j\quad \mbox{for}~1\leq j\leq n \quad\mbox{and}\quad\iota(g'_i)=g_i\quad \mbox{for}~1\leq i\leq n-1,
\end{equation*}
is injective. Meanwhile, by \cite[(3.6)]{ChP1}, there exists a surjective algebra homomorphism $\pi: \widehat{Y}_{r,n}\rightarrow Y_{r,n},$ which is given by
\begin{equation*}
\pi(t_j)=t'_j, \quad \pi(g_i)=g'_i,\quad \pi(X_1)=1
\end{equation*}
for $1\leq j\leq n$ and $1\leq i\leq n-1$.
\end{remark}

The elements $g_{i}$ in $\widehat{Y}_{r,n}$ are invertible with the inverse given by
\begin{equation*}\label{inverse}
g_{i}^{-1}=q^{-1}g_{i}-(1-q^{-1})e_{i}\quad\mbox{for }1\leq i\leq n-1.
\end{equation*}
Let $w\in \mathfrak{S}_{n}$ and let $w=s_{i_1}\cdots s_{i_{r}}$ be a \emph{reduced expression} of $w.$  By Matsumoto's theorem (see, e.g., \cite[Theorem 1.2.2]{GP}), the element $g_{w} :=g_{i_1}g_{i_2}\cdots g_{i_{r}}$ does not depend on the choice of the reduced expression of $w.$ For each $w\in\mathfrak{S}_n$, we denote by $\ell(w)$ the length of $w$ with respect to the simple reflections in $\mathfrak{S}_{n}.$ Then for any $w\in\mathfrak{S}_n$ and $1\leq i\leq n-1,$ we have
\begin{align}\label{multi-formula-ac-2019add}
g_{w}g_{s_{i}}=\begin{cases}g_{ws_{i}}& \hbox {if } \ell(ws_{i})>\ell(w), \\qg_{ws_{i}}+(q-1)g_{w}e_{i}& \hbox {if } \ell(ws_{i})<\ell(w).\end{cases}
\end{align}


Note that the elements $e_{i}$ are idempotents in $\widehat{Y}_{r,n}$. For any $1\leq i, k\leq n,$ we set
\begin{equation*}
e_{i,k} :=\frac{1}{r}\sum\limits_{s=0}^{r-1}t_{i}^{s}t_{k}^{-s}.
\end{equation*}
It is clear that $e_{i,i}=1,$ $e_{i,k}=e_{k,i}$ and that $e_{i,i+1}=e_{i}.$ It can be easily checked that the following holds:
\begin{equation*}\label{egge}
g_{i}e_{j,k}=e_{s_{i}(j),s_{i}(k)}g_{i}\quad\mbox{for $1\leq i\leq n-1$ and $1\leq j,k\leq n$}.
\end{equation*}
In particular, we have $g_{i}e_{i}=e_{i}g_{i}$ for all $1\leq i\leq n-1.$

We define the elements $X_{2},\ldots,X_{n}$ in $\widehat{Y}_{r,n}$ by
\begin{equation*}
X_{i+1} :=q^{-1}g_{i}X_{i}g_{i}\quad\text{for }1\leq i\leq n-1.\label{X2n}
\end{equation*}
It is proved in \cite[Lemma 1]{ChP1} that we have, for any $1\leq i\leq n-1$,
\begin{equation}
g_{i}X_{j}=X_{j}g_{i}\quad\mathrm{for}~1\leq j\leq n~\mathrm{such~that}~j\neq i, i+1.\label{giXj}
\end{equation}
Moreover, by \cite[Proposition 1]{ChP1}, we have that the elements $t_{1},\ldots, t_{n}, X_{1},\ldots, X_{n}$ form a commutative set, that is,
\begin{equation}\label{xyyx}
xy=yx\quad\mathrm{for~any}~x,y\in \{t_{1},\ldots, t_{n}, X_{1},\ldots, X_{n}\}.
\end{equation}
We shall often use the following identities (see \cite[Lemma 2.3]{ChP2}): for $1\leq i\leq n-1$,
\begin{equation}\label{gxxg}\begin{array}{rclcl}
g_{i}X_{i}\hspace*{-7pt}&=&\hspace*{-7pt}X_{i+1}g_{i}-(q-1)e_{i}X_{i+1}, \\[0.3em]
g_{i}X_{i+1}\hspace*{-7pt}&=&\hspace*{-7pt}X_{i}g_{i}+(q-1)e_{i}X_{i+1}, \\[0.3em]
g_{i}X_{i}^{-1}\hspace*{-7pt}&=&\hspace*{-7pt}X_{i+1}^{-1}g_{i}+(q-1)e_{i}X_{i}^{-1}, \\[0.3em]
g_{i}X_{i+1}^{-1}\hspace*{-7pt}&=&\hspace*{-7pt}X_{i}^{-1}g_{i}-(q-1)e_{i}X_{i}^{-1}.
\end{array}
\end{equation}

\subsection{The center of affine Yokonuma-Hecke algebras}
In the rest of this paper, we always assume that all tensor products of algebras or modules are taken over $\mathbb{K}$ unless otherwise stated.

Let $\mathcal{T}$ be the subalgebra of $\widehat{Y}_{r,n}$ generated by $t_{1},\ldots,t_{n}.$ Set $\Z_{r} :=\{0,1,\ldots,r-1\}.$ For each $\beta=(\beta_1,\ldots, \beta_n)\in\Z_{r}^n$, set $ t^\beta := t_{1}^{\beta_1}\cdots t_{n}^{\beta_n}$. Observe that the symmetric group $\mathfrak{S}_n$ acts naturally on $\mathcal{T}$ by permutations, which is given by $h\mapsto {}^wh$ for any $w\in \mathfrak{S}_n$ and $h\in \mathcal{T}.$ Then for $\beta=(\beta_1,\ldots, \beta_n)\in\Z_{r}^n$ and each $w\in \mathfrak{S}_n,$ we have ${}^w(t^{\beta})=t^{w\beta}$, where $w\beta=(\beta_{w^{-1}(1)},\ldots,\beta_{w^{-1}(n)}).$

Let $P_{n}$ be the subalgebra of $\widehat{Y}_{r,n}$ generated by $X_{1}^{\pm1},\ldots,X_{n}^{\pm1}.$ For each $\alpha=(\alpha_1, \ldots , \alpha_n)\in\Z^n,$ set $ X^\alpha := X_{1}^{\alpha_1}\cdots X_{n}^{\alpha_n}$. There exists a natural action of the symmetric group $\mathfrak{S}_n$ on $P_{n}$ by permutations. Let us denote this action by $f\mapsto {}^wf$ for any $w\in \mathfrak{S}_n$ and $f\in P_{n}.$ Then for $\alpha=(\alpha_1,\ldots,\alpha_n)\in\mathbb{Z}^n$ and each $w\in \mathfrak{S}_n$, we have ${}^w(X^{\alpha})=X^{w\alpha}$, where $w\alpha=(\alpha_{w^{-1}(1)},\ldots,\alpha_{w^{-1}(n)}).$

By \eqref{gxxg} and by induction, we can easily get the following lemma.
\begin{lemma}\label{commutator}
For any $f\in P_{n}$ and $1\leq i\leq n-1,$ we have
\begin{equation}\label{gif}
g_{i}f-{}^{s_{i}}fg_{i}=(q-1)e_{i}\frac{f-{}^{s_{i}}f}{1-X_{i}X_{i+1}^{-1}}.
\end{equation}
Note that $f-{}^{s_{i}}f$ is divisible by $1-X_{i}X_{i+1}^{-1},$ and hence $\frac{f-{}^{s_{i}}f}{1-X_{i}X_{i+1}^{-1}},$ as an element of the field of fractions of $P_{n},$ lies in $P_{n}.$
\end{lemma}

Let $\preceq$ denote the Bruhat order on $\mathfrak{S}_{n}.$ By Lemma \ref{commutator}, we can easily get the next lemma.
\begin{lemma}\label{gtX-commutator}
Let $w\in \mathfrak{S}_{n},$ $t\in \mathcal{T}$ and $\alpha=(\alpha_1,\ldots,\alpha_{n})\in \mathbb{Z}^{n}.$ Then in $\widehat{Y}_{r,n},$ we have
\begin{equation*}
g_{w}tX^{\alpha}=({}^{w}t)X^{w\alpha}g_{w}+\sum\limits_{u\preceq w; u\neq w}t_{u}f_{u}g_{u},\quad\text{and}\quad
tX^{\alpha}g_{w}=g_{w}({}^{w^{-1}}t)X^{w^{-1}\alpha}+\sum\limits_{u\preceq w; u\neq w}g_{u}t_{u}'f_{u}'
\end{equation*}
for some $f_{u}, f_{u}'\in P_{n}$ and $t_{u}, t_{u}'\in \mathcal{T}.$
\end{lemma}

The following theorem gives a PBW basis of the affine Yokonuma-Hecke algebra $\widehat{Y}_{r,n}.$
\begin{theorem}\label{PBW}{\rm (See \cite[Theorem 4.4]{ChP2}.)}
The elements $\{X^{\alpha}t^\beta g_{w}\:|\: \alpha\in \mathbb{Z}^{n}, \beta\in \mathbb{Z}_r^n\text{ and }w\in \mathfrak{S}_{n}\}$ form a $\mathbb{K}$-basis of $\widehat{Y}_{r,n}.$
\end{theorem}

Let $G$ be a cyclic group of order $r$ and set $T :=G^n.$ By Theorem \ref{PBW}, the subalgebra $\mathcal{T}$ can be identified with the group algebra $\mathbb{K}T$ of the group $T$ while the subalgebra $P_n$ can be identified with the algebra $\mathbb{K}[X_1^{\pm1},\ldots,X_n^{\pm1}]$ of Laurent polynomials in $X_1,\ldots, X_{n}.$
Let $P_{n}(T)$ be the subalgebra of $\widehat{Y}_{r,n}$ generated by $t_{1},\ldots,t_{n}$ and $X_{1}^{\pm1},\ldots,X_{n}^{\pm1}.$ Then we have $$P_{n}(T)\cong P_{n}\otimes\mathbb{K}T.$$

\begin{lemma}
The center $Z(\widehat{Y}_{r,n})$ of $\widehat{Y}_{r,n}$ is contained in the subalgebra $P_{n}(T).$
\end{lemma}
\begin{proof}
Assume that $z$ is a central element of $\widehat{Y}_{r,n}.$ By Theorem \ref{PBW}, we can write $z$ as $z=\sum_{w\in \mathfrak{S}_{n}}z_{w}g_{w},$ where $z_{w}=\sum d_{t,\alpha}X^{\alpha}t\in P_{n}(T).$ Take $\tau$ to be maximal under the Bruhat order such that $z_{\tau}\neq 0.$ Assume that $\tau\neq 1.$ Then there exists some $i\in \{1,2,\ldots,n\}$ satisfying $\tau(i)\neq i.$

By Lemma \ref{gtX-commutator}, we have $$0=X_{i}z-zX_{i}=z_{\tau}(X_{i}-X_{\tau(i)})g_{\tau}+\sum\limits_{u\preceq \tau; u\neq \tau}a_{t',\beta,u}X^{\beta}t'g_{u}.$$ By Theorem \ref{PBW}, we see that $z_{\tau}=0,$ which contradicts the choice of $\tau.$ Thus, $\tau=1$ and $z\in P_{n}(T).$
\end{proof}

We set
$$P_{n}(T)^{\mathfrak{S}_{n}} :=\Big\{\sum d_{\alpha,\beta}X^{\alpha}t^{\beta}\in  P_{n}(T)\:\big|\:\sum d_{\alpha,\beta}X^{\alpha}t^{\beta}=\sum d_{\alpha,\beta}X^{w\alpha}t^{w\beta}~\mathrm{for~any}~w\in \mathfrak{S}_{n}\Big\}.$$

\begin{theorem}\label{center}
We have $Z(\widehat{Y}_{r,n})=P_{n}(T)^{\mathfrak{S}_{n}}.$
\end{theorem}
\begin{proof}
Suppose that
$$z=\sum\limits_{\alpha\in \mathbb{Z}^{n}; \beta\in \mathbb{Z}_{r}^{n}}d_{\alpha,\beta}X^{\alpha}t^{\beta}\in Z(\widehat{Y}_{r,n}).$$
Then for each $1\leq k\leq n-1,$ we have $g_{k}z=zg_{k},$ that is, $g_{k}\cdot \sum d_{\alpha,\beta}X^{\alpha}t^{\beta}=$
$\sum d_{\alpha,\beta}X^{\alpha}t^{\beta}g_{k}.$ Thus, by Lemma \ref{commutator} we get $$\sum_{\alpha,\beta} d_{\alpha,\beta} X^{s_{k}\alpha}t^{s_{k}\beta}g_{k}+(q-1)\sum_{\alpha,\beta}d_{\alpha,\beta}e_{k}\frac{X^{\alpha}-X^{s_{k}\alpha}}{1-X_{k}X_{k+1}^{-1}}t^{\beta}=
\sum_{\alpha,\beta}d_{\alpha,\beta}X^{\alpha}t^{\beta}g_{k}.$$
By Theorem \ref{PBW}, we must have
\begin{align}
\sum_{\alpha,\beta}d_{\alpha,\beta}X^{s_{k}\alpha}t^{s_{k}\beta}=\sum_{\alpha,\beta}d_{\alpha,\beta}X^{\alpha}t^{\beta} \quad\text{ for any }1\leq k\leq n-1,\label{center-eqn1}\\
\sum_{\alpha,\beta}d_{\alpha,\beta}e_{k}\frac{X^{\alpha}-X^{s_{k}\alpha}}{1-X_{k}X_{k+1}^{-1}}t^{\beta}=0 \quad\text{ for any }1\leq k\leq n-1.\label{center-eqn2}
\end{align}

We claim that \eqref{center-eqn1} implies \eqref{center-eqn2}. Assume that \eqref{center-eqn1} holds. For each $\beta\in\Z_r^n$ and $1\leq k\leq n-1$, by \cite[(2.13)]{ChP1} we can easily get $e_kt^\beta=e_k t^{s_k\beta}.$
Then we have
\begin{align*}
\sum d_{\alpha,\beta}e_kX^{\alpha}t^{\beta}&=\sum d_{\alpha,\beta}e_kX^{s_k\alpha}t^{s_k\beta} \text{ (by }\eqref{center-eqn1}) \\
&=\sum d_{\alpha,\beta}X^{s_k\alpha}e_kt^{s_k\beta} \text{ (by }\eqref{xyyx}) \\
&= \sum d_{\alpha,\beta}X^{s_k\alpha}e_kt^{\beta} \\
&= \sum d_{\alpha,\beta}e_kX^{s_k\alpha}t^{\beta}.
\end{align*}
Therefore, we see that \eqref{center-eqn2} holds.

By using Theorem \ref{PBW} again, we see that \eqref{center-eqn1} holds if and only if $d_{\alpha,\beta}=d_{s_k\alpha,s_k\beta}$
for $\alpha\in\Z^n, \beta\in\Z_r^n$ and $1\leq k\leq n-1,$ or equivalently, $d_{w\alpha,w\beta}=d_{\alpha,\beta}$ for any $w\in\mathfrak{S}_n$ and $\alpha\in\Z^n, \beta\in\Z_r^n.$

By reversing the above arguments, it is easy to see that an element $z\in\widehat{Y}_{r,n}$ of the form $\sum_{\alpha,\beta} d_{\alpha,\beta}X^{\alpha}t^{\beta}$ with $d_{w\alpha,w\beta}=d_{\alpha,\beta}$ for any $w\in\mathfrak{S}_n,$ belongs to $Z(\widehat{Y}_{r,n}).$
\end{proof}
\begin{corollary}
If $M$ is a simple $\widehat{Y}_{r,n}$-module, then $M$ is finite dimensional.
\end{corollary}
\begin{proof}
It is known that $P_{n}$ is a free $\mathbb{K}[X_{1}^{\pm1},\ldots,X_{n}^{\pm1}]^{\mathfrak{S}_{n}}$-module of finite rank $n!,$ and $\mathbb{K}T$ is a free $(\mathbb{K}T)^{\mathfrak{S}_{n}}$-module of finite rank. By Theorem \ref{center} we conclude that $\widehat{Y}_{r,n}$ is a finitely generated module over its center $Z(\widehat{Y}_{r,n}).$ Dixmier's version of Schur's lemma (see, e.g., \cite[0.5.2]{Wa}) implies that the center of $\widehat{Y}_{r,n}$ acts by scalars on absolutely simple modules, which implies that $M$ is a simple module for a finite dimensional algebra, and hence $M$ is finite dimensional.
\end{proof}

\begin{remark} Recently, Chlouveraki and S\'{e}cherre \cite[Theorem 4.3]{ChS} proved that the affine Yokonuma-Hecke algebra is a particular case of the pro-$p$-Iwahori-Hecke algebra defined by Vign\'eras in \cite{Vi1}. In \cite[Theorem 1.3]{Vi2} Vign\'eras described the center of the pro-$p$-Iwahori-Hecke algebra over any commutative ring $R.$ Thus, our Theorem \ref{center} can be regarded as a particular case of Vign\'eras' results.
\end{remark}

\section{An equivalence of two categories}

In this section, we establish an explicit equivalence between the category $\widehat{Y}_{r,n}$-$\mathbf{mod}$ of finite dimensional $\widehat{Y}_{r,n}$-modules and the category $\widehat{\mathcal{H}}_{r,n}$-$\mathbf{mod}$ of finite dimensional $\widehat{\mathcal{H}}_{r,n}$-modules, where $\widehat{\mathcal{H}}_{r,n}$ is a direct sum of tensor products of various affine Hecke algebras of type $A.$ The category equivalence plays a key role throughout the rest of this paper.

\subsection{A decomposition of $\widehat{Y}_{r,n}$-modules}
Recall that $G$ is a cyclic group of order $r$ and $T=G^n.$ Since $\mathbb{K}$ is an algebraically closed field of characteristic $p$ such that $p$ does not divide $r$, there exists a primitive $r$-th root of unity $\zeta.$ Fix a generator $g$ of the cyclic group $G.$ For each $1\leq a\leq r,$ there exists a one-dimensional $\mathbb{K}G$-module $V_a$ such that $g$ acts on $V_{a}$ as the scalar $\zeta^{a}.$ Since $\mathbb{K}$ is an algebraically closed field of characteristic $p$ such that $p$ does not divide $r,$ the set $\{V_1,\ldots,V_{r}\}$ is a complete one of pairwise non-isomorphic finite dimensional simple $\mathbb{K}G$-modules. Hence, we see that $\{V_{i_1}\otimes\cdots\otimes V_{i_{n}}\:|\:1\leq i_{1},\ldots,i_{n}\leq r\}$ is a complete set of pairwise non-isomorphic simple $\mathbb{K}T$-modules.

By the above arguments, we can easily obtain the following lemma, which can be regarded as a particular case of \cite[Corollary 3.3]{WW}. Recall that $e_{i}=\frac{1}{r}\sum_{s=0}^{r-1}t_{i}^{s}t_{i+1}^{-s}$ for $1\leq i\leq n-1$.


\begin{lemma}\label{weig-e-0}
Suppose that $1\leq i_{1},\ldots,i_{n}\leq r$ and $1\leq k\leq n-1.$ If $i_{k}=i_{k+1},$ $e_{k}$ acts as the identity on the module $V_{i_1}\otimes\cdots\otimes V_{i_{n}};$ otherwise, $e_{k}$ acts as zero on it.
\end{lemma}
\begin{proof}
 The lemma easily follows from the fact that $e_{1}$ acts as zero on a simple $\mathbb{K}G^{2}$-module $V_{k}\otimes V_{l}$,where $1\leq k, l\leq r$ and $k\neq l,$ and it acts as the identity on the $\mathbb{K}G^{2}$-module $V_{k}^{\otimes2}$ for $1\leq k\leq r.$
\end{proof}

Given an algebra $S,$ we denote by $S$-$\mathbf{mod}$ the category of finite dimensional left $S$-modules. Since $\mathbb{K}$ is an algebraically closed field of characteristic $p$ such that $p$ does not divide $r$, every module in $\widehat{Y}_{r, n}$-$\mathbf{mod}$ is semisimple when restricted to the subalgebra $\mathbb{K}T.$

Let $\mathcal{C}_{r}(n)$ be the set of $r$-\emph{compositions} of $n,$ that is, the set of $r$-tuples of nonnegative integers $\mu=(\mu_{1},\ldots,\mu_{r})$ such that $\sum_{1\leq a\leq r}\mu_{a}=n.$ Fix one $\mu=(\mu_{1},\ldots,\mu_{r})\in \mathcal{C}_{r}(n).$ Let $$V(\mu) :=V_{1}^{\otimes \mu_{1}}\otimes\cdots\otimes V_{r}^{\otimes \mu_{r}}$$ be the associated simple $\mathbb{K}T$-module and let $\mathfrak{S}_{\mu} :=\mathfrak{S}_{\mu_{1}}\times\cdots\times\mathfrak{S}_{\mu_{r}}$ be the associated \emph{Young subgroup} of $\mathfrak{S}_{n}.$ For each $1\leq i\leq r,$ since $\dim V_i=1,$ we will assume that $V_i=\mathbb{K}v_i.$ Set $v_\mu :=v_1^{\otimes \mu_1}\otimes\cdots\otimes v_r^{\otimes\mu_r}\in V(\mu).$ Then we have $V(\mu)=\mathbb{K}v_\mu$.

Fix one $\mu=(\mu_{1},\ldots,\mu_{r})\in \mathcal{C}_{r}(n).$ We denote by $\mathcal{O}(\mu)$ a complete set of left coset representatives of $\mathfrak{S}_{\mu}$ in $\mathfrak{S}_{n}.$ We define $\widehat{Y}_{r,\mu}$ to be the subalgebra of $\widehat{Y}_{r,n}$ generated by $t_{1},\ldots,t_{n},$ $X_{1}^{\pm1},\ldots,X_{n}^{\pm1}$ and $g_{w}$ for all $w\in \mathfrak{S}_{\mu}.$ Then by Definition \ref{rel-def-Y1} we have
\begin{equation}\label{eq:Young-sub}
\widehat{Y}_{r,\mu}\cong\widehat{Y}_{r,\mu_{1}}\otimes\cdots\otimes\widehat{Y}_{r,\mu_{r}}.
\end{equation}
Moreover, every module in $\widehat{Y}_{r,\mu}$-$\mathbf{mod}$ is semisimple when restricted to the subalgebra $\mathbb{K}T$.

For each $\mu\in \mathcal{C}_{r}(n)$ and $M\in \widehat{Y}_{r, n}$-$\mathbf{mod},$ we define $I_{\mu}M$ to be the \emph{isotypical subspace} of $V(\mu)$ in $M,$ that is, the sum of all simple $\mathbb{K}T$-submodules of $M$ isomorphic to $V(\mu).$ We define $M_{\mu}$ by
\begin{equation}\label{isotypical-subspaces}
M_{\mu} :=\sum_{w\in \mathfrak{S}_{n}}g_{w}(I_{\mu}M).
\end{equation}
In general, for two algebras $A\subseteq B$ and an $A$-module $M,$ we then define the induced $B$-module by $\mathrm{Ind}_{A}^{B}M:=B\otimes_{A}M.$

\begin{lemma}\label{weig-e-1}
Assume that $\mu=(\mu_{1},\ldots,\mu_{r})\in \mathcal{C}_{r}(n)$ and $M\in \widehat{Y}_{r,n}$-$\mathbf{mod}.$ Then, $I_{\mu}M$ is a $\widehat{Y}_{r,\mu}$-submodule and $M_{\mu}$ is a $\widehat{Y}_{r,n}$-submodule of $M.$ Moreover, we have $M_{\mu}\cong \mathrm{Ind}_{\widehat{Y}_{r,\mu}}^{\widehat{Y}_{r,n}}(I_{\mu}M).$
\end{lemma}
\begin{proof}
Since $X_{i}^{\pm 1}$ commutes with $\mathbb{K}T$ for each $1\leq i\leq n,$ we see that each $X_{i}^{\pm 1}$ ($1\leq i\leq n$) maps a simple $\mathbb{K}T$-submodule of $M$ to an isomorphic one. Hence, $I_{\mu}M$ is invariant under the action of the subalgebra $P_{n}.$ Fix one $i\in \{1,2,\ldots,n-1\}\backslash\{\mu_{1},\mu_1+\mu_2,\ldots,\mu_1+\cdots+\mu_{r-1}\}.$ For each $1\leq j\leq n,$ we have $t_{j}(g_{i}^{\pm 1}v_\mu)=g_{i}^{\pm 1}t_{s_{i}(j)}(v_\mu)=g_{i}^{\pm 1}t_j(v_\mu).$ From the above identities we see that $g_{i}V(\mu)$ is a $\mathbb{K}T$-module, and is isomorphic to $V(\mu).$ Hence, for each $w\in \mathfrak{S}_{\mu},$ $g_{w}$ maps a simple $\mathbb{K}T$-submodule of $M$ isomorphic to $V(\mu)$ to another isomorphic one. Thus, $I_{\mu}M$ is invariant under the action of $g_{w}$ for all $w\in \mathfrak{S}_{\mu}.$ Since $\widehat{Y}_{r,\mu}$ is generated by $P_{n},$ $\mathbb{K}T$ and $g_{w}$ (with $w\in \mathfrak{S}_{\mu}$), we see that $I_{\mu}M$ is a $\widehat{Y}_{r,\mu}$-submodule.


By \eqref{gxxg} and \eqref{isotypical-subspaces}, it is easy to see that $M_{\mu}$ is a $\widehat{Y}_{r,n}$-submodule of $M.$

By Frobenius reciprocity, we have a nonzero $\widehat{Y}_{r,n}$-homomorphism
$$\varphi: \mathrm{Ind}_{\widehat{Y}_{r,\mu}}^{\widehat{Y}_{r,n}}(I_{\mu}M)\rightarrow M_{\mu}.$$
Note that by \eqref{multi-formula-ac-2019add}, for each $w=\tau w_{\mu}$ with $\tau\in \mathcal{O}(\mu)$ and $w_{\mu}\in \mathfrak{S}_{\mu},$ we have $g_{w}(I_{\mu}M)=g_{\tau}(I_{\mu}M).$ Hence
\begin{equation}\label{isotypical-subspaces-add}
M_{\mu}=\sum_{\tau\in \mathcal{O}(\mu)}g_{\tau}(I_{\mu}M),
\end{equation}
which implies that $\varphi$ is surjective. Observe that as $\mathbb{K}T$-modules, $g_{\tau_1}(I_{\mu}M)$ and $g_{\tau_2}(I_{\mu}M)$ are isotypical subspaces of two non-isomorphic simple $\mathbb{K}T$-modules in $M$ if $\tau_1\neq \tau_2\in \mathcal{O}(\mu).$ Hence the sum in \eqref{isotypical-subspaces-add} is in fact direct, and the dimension of $M_{\mu}$ is $$\dim{I_{\mu}M}\cdot \frac{n!}{\mu_{1}!\mu_{2}!\cdots \mu_{r}!},$$ where $\dim{I_{\mu}M}$ denotes the dimension of $I_{\mu}M.$ Thus, $\varphi$ is an isomorphism by comparing dimensions of both sides of it.
\end{proof}

\begin{lemma}\label{weig-e-2}
For each $M\in\widehat{Y}_{r,n}$-$\mathbf{mod}$, we have the following decomposition$:$
\begin{equation}\label{the-decomposition-aaa}
M=\bigoplus_{\mu\in \mathcal{C}_{r}(n)}M_{\mu}.
\end{equation}
\end{lemma}
\begin{proof}
By \eqref{isotypical-subspaces-add} and the definition of $I_{\mu}M$, we see that for each $\mu=(\mu_{1},\ldots,\mu_{r})\in \mathcal{C}_{r}(n),$
$M_{\mu}$ is the sum of all simple $\mathbb{K}T$-submodules of $M$ isomorphic to $V_{j_1}\otimes\cdots\otimes V_{j_{n}}$ for all $j_1, \ldots, j_n$ such that $\#\{1\leq k\leq n\:|\:j_{k}=i\}=\mu_{i}$ for each $1\leq i\leq r.$ Note that $M$ is semisimple as a $\mathbb{K}T$-module. This, together with the fact that $\{V_{i_1}\otimes\cdots\otimes V_{i_{n}}\:|\:1\leq i_{1},\ldots,i_{n}\leq r\}$ is a complete set of pairwise non-isomorphic simple $\mathbb{K}T$-modules, gives rise to the decomposition \eqref{the-decomposition-aaa}.
\end{proof}
\subsection{An equivalence of two module categories}
Recall the description of affine Hecke algebras due to J. Bernstein (see [Lu1-2]). For each $m\in \mathbb{Z}_{\geq 1},$ the \emph{extended affine Hecke algebra} $\widehat{\mathcal{H}}_{m}$ of type $A$ is a $\mathbb{K}$-algebra generated by elements $T_{i},$ $Y_{j}^{\pm 1}$ (with $1\leq i\leq m-1$ and $1\leq j\leq m$) subject to the following relations: \vskip1mm

(1) $(T_{i}+1)(T_{i}-q)=0~ \text{ for }1\leq i\leq m-1;$

(2) $T_{i}T_{i+1}T_{i}=T_{i+1}T_{i}T_{i+1} ~ \text{ for }1\leq i\leq m-2;$

(3) $T_{i}T_{j}=T_{j}T_{i}~ \text{ for }1\leq i,j\leq m-1$ such that $|i-j|\geq 2;$

(4) $Y_{i}Y_{i}^{-1}=Y_{i}^{-1}Y_{i}=1,\quad Y_{i}Y_{j}=Y_{j}Y_{i} ~\text{ for}~1\leq i,j\leq m;$

(5) $T_{i}Y_{i}T_{i}=q Y_{i+1}~\text{ for}~1\leq i\leq m-1;$

(6) $T_{i}Y_{j}=Y_{j}T_{i}~\text{ for}~1\leq i\leq m-1,$ $1\leq j\leq m$ such that $j\neq i, i+1.$\vskip1mm

We assume that $\widehat{\mathcal{H}}_{0}=\mathbb{K}.$ Let $w\in \mathfrak{S}_{m}$ and let $w=s_{i_1}\cdots s_{i_{k}}$ be a reduced expression of $w.$ It is well known that the element $T_{w} :=T_{i_1}T_{i_2}\cdots T_{i_{k}}$ is well-defined.

For each $\mu=(\mu_{1},\ldots,\mu_{r})\in \mathcal{C}_{r}(n),$ we denote by $\widehat{\mathcal{H}}_{r,\mu}$ the subalgebra of $\widehat{\mathcal{H}}_{n}$ generated by $T_{i},$ $Y_{j}^{\pm 1}$ with $i\in \{1,2,\ldots,n-1\}\backslash\{\mu_{1},\mu_1+\mu_2,\ldots,\mu_1+\cdots+\mu_{r-1}\}$ and $1\leq j\leq n$. Note that $\widehat{\mathcal{H}}_{r,\mu}$ is naturally isomorphic to the tensor product $\widehat{\mathcal{H}}_{\mu_{1}}\otimes\cdots\otimes\widehat{\mathcal{H}}_{\mu_{r}}$ and we shall identify them. We define the following algebra:
$$\widehat{\mathcal{H}}_{r,n} :=\bigoplus_{\mu\in \mathcal{C}_{r}(n)}\widehat{\mathcal{H}}_{r,\mu}.$$

In the following we will see that the strategy in \cite[Section 3]{WW} can also be applied to our setting, even though the quadratic relations for the generators $g_i$ of $\widehat{Y}_{r,n}$ in \eqref{rel-def-Y1} look quite different from \cite[(2.4-2.5)]{WW}. The key observation is that the elements $e_k$ act as either zero or the identity on a simple $\mathbb{K}T$-module due to Lemma \ref{weig-e-0}.

\begin{proposition}\label{weig-e-3}
Assume that $\mu=(\mu_{1},\ldots,\mu_{r})\in \mathcal{C}_{r}(n)$ and $N\in \widehat{Y}_{r,\mu}$-$\mathbf{mod}.$ Then $\mathrm{Hom}_{\mathbb{K}T}(V(\mu), N)$ is an $\widehat{\mathcal{H}}_{r,\mu}$-module with the action given by
\begin{equation}\label{tw-vmu-yk-vum}\begin{array}{rcl}
(T_{w}\diamond \phi)(v_\mu)\hspace*{-7pt}&=&\hspace*{-7pt}g_{w}\phi(v_\mu),\\[0.4em]
(Y_{k}^{\pm1}\diamond \phi)(v_\mu)\hspace*{-7pt}&=&\hspace*{-7pt}X_{k}^{\pm1}\phi(v_\mu)
\end{array}
\end{equation}
for $w\in \mathfrak{S}_{\mu},$ $1\leq k\leq n$ and $\phi\in \mathrm{Hom}_{\mathbb{K}T}(V(\mu), N).$ Thus, $\mathrm{Hom}_{\mathbb{K}T}(V(\mu), -)$ is a functor from $\widehat{Y}_{r,\mu}$-$\mathbf{mod}$ to $\widehat{\mathcal{H}}_{r,\mu}$-$\mathbf{mod}.$
\end{proposition}
\begin{proof}
We first show that $T_{w}\diamond \phi$ is a $\mathbb{K}T$-module homomorphism. It suffices to consider each $T_{i}\diamond \phi$ for $i\in \{1,2,\ldots,n-1\}\backslash\{\mu_{1},\mu_1+\mu_2,\ldots,\mu_1+\cdots+\mu_{r-1}\}.$ Observe that  we have, for each $1\leq j\leq n,$
\begin{align*}
(T_{i}\diamond \phi)(t_{j}(v_\mu))&=(T_{i}\diamond \phi)(t_{s_{i}(j)}(v_\mu))\\
&=g_{i}\phi(t_{s_{i}(j)}(v_\mu))\\
&=g_{i}t_{s_{i}(j)}\phi(v_\mu)\\
&=t_{j}(T_{i}\diamond \phi)(v_\mu).
\end{align*}
The fact that $Y_{k}^{\pm1}\diamond \phi$ is a $\mathbb{K}T$-module homomorphism can be proved similarly.

By Lemma \ref{weig-e-0}, we see that $e_{i},$ for each $i\neq \mu_1,\mu_1+\mu_2,\ldots,\mu_1+\cdots+\mu_{r-1}$, acts on $V(\mu)$ as the identity. Then it is easy to verify that the actions given in \eqref{tw-vmu-yk-vum} satisfy the relations for the $\widehat{\mathcal{H}}_{r,\mu}$-module structure on $\mathrm{Hom}_{\mathbb{K}T}(V(\mu), N).$ We skip the details and leave them to the reader.
\end{proof}

\begin{proposition}\label{weig-e-4}
Assume that $\mu=(\mu_{1},\ldots,\mu_{r})\in \mathcal{C}_{r}(n)$ and $M$ is a finite dimensional $\widehat{\mathcal{H}}_{r,\mu}$-module. Then $V(\mu)\otimes M$ affords a $\widehat{Y}_{r,\mu}$-module via
\begin{equation}\label{tk-gw-xkvmu-vum}\begin{array}{rclcl}
t_{k}\ast(v_\mu\otimes z)\hspace*{-7pt}&=&\hspace*{-7pt}t_{k}(v_\mu)\otimes z, \\[0.4em]
g_{w}\ast(v_\mu\otimes z)\hspace*{-7pt}&=&\hspace*{-7pt}v_\mu\otimes T_{w}z, \\[0.4em]
X_{k}^{\pm1}\ast(v_\mu\otimes z)\hspace*{-7pt}&=&\hspace*{-7pt}v_\mu\otimes Y_{k}^{\pm1}z
\end{array}
\end{equation}
for $1\leq k\leq n,$ $w\in \mathfrak{S}_{\mu}$ and $z\in M.$ There exists an isomorphism of $\widehat{\mathcal{H}}_{r,\mu}$-modules $\Phi: M\rightarrow \mathrm{Hom}_{\mathbb{K}T}(V(\mu), V(\mu)\otimes M)$ given by $\Phi(z)(v_\mu)=v_\mu\otimes z.$ Moreover, $V(\mu)\otimes M$ is a simple $\widehat{Y}_{r,\mu}$-module if and only if $M$ is a simple $\widehat{\mathcal{H}}_{r,\mu}$-module.
\end{proposition}
\begin{proof}
In order to verify that $V(\mu)\otimes M$ is a $\widehat{Y}_{r,\mu}$-module with the actions of its generators given in \eqref{tk-gw-xkvmu-vum}, we need to check that the actions satisfy the relations listed in \eqref{rel-def-Y1} and \eqref{rel-def-Y333-daadaaa}. We only check the quadratic relation for each $g_{i}$ with $i\in \{1,2,\ldots,n-1\}\backslash\{\mu_{1},\mu_1+\mu_2,\ldots,\mu_1+\cdots+\mu_{r-1}\}$ and leave the remaining verifications to the reader.

For each $z\in M,$ we have $g_{i}^{2}\ast(v_\mu\otimes z)=v_\mu\otimes T_{i}^{2}z=v_\mu\otimes (q+(q-1)T_{i})z,$ while
$(q+(q-1)e_{i}g_{i})\ast(v_\mu\otimes z)=qv_\mu\otimes z+(q-1)e_{i}v_\mu\otimes T_{i}z.$
Note that $e_{i}v_\mu=v_\mu$ by Lemma \ref{weig-e-0} and then the quadratic relation in \eqref{rel-def-Y1} holds for each $g_{i}$ with $i\in \{1,2,\ldots,n-1\}\backslash\{\mu_{1},\mu_1+\mu_2,\ldots,\mu_1+\cdots+\mu_{r-1}\}.$

It is easy to see that $\Phi$ is a well-defined injective $\widehat{\mathcal{H}}_{r,\mu}$-module homomorphism. Moreover, as a $\mathbb{K}T$-module, $V(\mu)\otimes M$ is isomorphic to a direct sum of finite copies of $V(\mu).$ Thus, $\Phi$ is an isomorphism by comparing dimensions of these two modules.

Assume that $V(\mu)\otimes M$ is a simple $\widehat{Y}_{r,\mu}$-module and $E$ is a nonzero $\widehat{\mathcal{H}}_{r,\mu}$-submodule of $M.$ Then $V(\mu)\otimes E$ is a nonzero $\widehat{Y}_{r,\mu}$-submodule of $V(\mu)\otimes M,$ which implies $E=M.$ Conversely, assume that $M$ is a simple $\widehat{\mathcal{H}}_{r,\mu}$-module and $P$ is a nonzero $\widehat{Y}_{r,\mu}$-submodule of $V(\mu)\otimes M.$ By Proposition \ref{weig-e-3}, $\mathrm{Hom}_{\mathbb{K}T}(V(\mu), P)$ is a nonzero $\widehat{\mathcal{H}}_{r,\mu}$-submodule of $\mathrm{Hom}_{\mathbb{K}T}(V(\mu), V(\mu)\otimes M)\cong M.$ Since $M$ is simple, we have $\mathrm{Hom}_{\mathbb{K}T}(V(\mu), P)\cong M.$ Note that $P,$ as a $\mathbb{K}T$-module, is isomorphic to a direct sum of finite copies of $V(\mu).$ Hence, we must have $P=V(\mu)\otimes M$ by a dimension comparison.
\end{proof}

\begin{proposition}\label{weig-e-5}
Assume that $N\in \widehat{Y}_{r,n}$-$\mathbf{mod}.$ Then for each $\mu\in \mathcal{C}_{r}(n),$
\begin{align*}
\Psi: V(\mu)\otimes \mathrm{Hom}_{\mathbb{K}T}(V(\mu), I_{\mu}N)&\longrightarrow I_{\mu}N,\\
v_\mu\otimes \psi&\mapsto \psi(v_\mu
)
\end{align*}
defines an isomorphism of $\widehat{Y}_{r,\mu}$-modules.
\end{proposition}
\begin{proof}
By Lemma \ref{weig-e-1}, $I_{\mu}N$ is a $\widehat{Y}_{r,\mu}$-module. It follows from Propositions \ref{weig-e-3} and \ref{weig-e-4} that $V(\mu)\otimes \mathrm{Hom}_{\mathbb{K}T}(V(\mu), I_{\mu}N)$ is a $\widehat{Y}_{r,\mu}$-module.

It is easy to check that $\Psi$ is a $\widehat{Y}_{r,\mu}$-homomorphism. Since $I_{\mu}N,$ as a $\mathbb{K}T$-module, is isomorphic to a direct sum of finite copies of $V(\mu),$ we see that $\Psi$ is surjective. Hence $\Psi$ is an isomorphism by a dimension comparison.
\end{proof}

Now we can prove one of the main results of this paper.
\begin{theorem}\label{weig-e-6}
The functor $\mathcal{F}: \widehat{Y}_{r,n}$-$\mathbf{mod}$ $\rightarrow$ $\widehat{\mathcal{H}}_{r,n}$-$\mathbf{mod}$ defined by
$$\mathcal{F}(N)=\bigoplus_{\mu\in \mathcal{C}_{r}(n)}\mathrm{Hom}_{\mathbb{K}T}(V(\mu), I_{\mu}N)$$
is an equivalence of categories with the inverse $\mathcal{G}: \widehat{\mathcal{H}}_{r,n}$-$\mathbf{mod}$ $\rightarrow$ $\widehat{Y}_{r,n}$-$\mathbf{mod}$ given by $$\mathcal{G}(\oplus_{\mu\in \mathcal{C}_{r}(n)}P_{\mu})=\bigoplus_{\mu\in \mathcal{C}_{r}(n)}\mathrm{Ind}_{\widehat{Y}_{r,\mu}}^{\widehat{Y}_{r,n}}(V(\mu)\otimes P_{\mu}).$$
\end{theorem}
\begin{proof}
Observe that the map $\Phi$ in Proposition \ref{weig-e-4} is natural in $M$ and $\Psi$ in Proposition \ref{weig-e-5} is natural in $N.$ Then one can easily verify that $\mathcal{F}\mathcal{G}\cong \mathrm{id}$ and $\mathcal{G}\mathcal{F}\cong \mathrm{id}$ by using Lemmas \ref{weig-e-1} and \ref{weig-e-2}, and Propositions \ref{weig-e-3}-\ref{weig-e-5}.
\end{proof}

\section{Simple $\widehat{Y}_{r,n}$-modules and modular branching rules for $\widehat{Y}_{r,n}$}
In this section, we shall present three applications of the equivalence of module categories established in Theorem \ref{weig-e-6}. We shall classify all finite dimensional simple $\widehat{Y}_{r,n}$-modules, and establish the modular branching rule for $\widehat{Y}_{r,n}$ which provides a description of the socle of the restriction to $\widehat{Y}_{r,(n-1,1)}$ of a simple $\widehat{Y}_{r,n}$-module. We also give a block decomposition of $\widehat{Y}_{r,n}$-$\mathbf{mod}.$

\subsection{Simple $\widehat{Y}_{r,n}$-modules}
\begin{theorem}\label{simple-mod-1}
Each simple $\widehat{Y}_{r,n}$-module is isomorphic to a module of the form
$$S_{\mu}(L.) :=\mathrm{Ind}_{\widehat{Y}_{r,\mu}}^{\widehat{Y}_{r,n}}\big(V(\mu)\otimes (L_{1}\otimes\cdots\otimes L_{r})\big),$$
where $\mu=(\mu_{1},\ldots,\mu_{r})\in \mathcal{C}_{r}(n)$ and $L_{k}$ $(1\leq k\leq r)$ is a simple $\widehat{\mathcal{H}}_{\mu_{k}}$-module. Moreover, the above modules $S_{\mu}(L.),$ with $\mu=(\mu_{1},\ldots,\mu_{r})$ running through $\mathcal{C}_{r}(n)$ and $L_{k}$ $(1\leq k\leq r)$ running through all non-isomorphic finite dimensional simple $\widehat{\mathcal{H}}_{\mu_{k}}$-modules, form a complete set of pairwise non-isomorphic simple $\widehat{Y}_{r,n}$-modules.
\end{theorem}
\begin{proof}
It follows from the category equivalence established in Theorem \ref{weig-e-6}.
\end{proof}

Let $e$ be the smallest positive integer such that $q^{e}=1;$ set $e=\infty$ if no such integer exists.

\begin{remark}
Ariki and Mathas have given the classification of the simple modules of an extended affine Hecke algebra of type $A$ over an arbitrary field in terms of aperiodic multisegments. In particular, the non-isomorphic simple $\widehat{\mathcal{H}}_{m}$-modules are indexed by the set $\mathcal{M}_{e}^{m}(\mathbb{K})$ (see \cite[Theorem B(i)]{AM} for more details). Combining this with Theorem \ref{simple-mod-1}, we see that the simple $\widehat{Y}_{r,n}$-modules are indexed by the following set $$\mathcal{A} :=\big\{(\mu, \psi_{1},\ldots,\psi_{r})\:|\:\mu=(\mu_{1},\ldots,\mu_{r})\in \mathcal{C}_{r}(n)\text{ and }\psi_{i}\in \mathcal{M}_{e}^{\mu_{i}}(\mathbb{K})\text{ for }1\leq i\leq r\big\}.$$
\end{remark}

\subsection{Modular branching rules for $\widehat{Y}_{r,n}$}
We refer the modular branching rule to the determination of simple $\widehat{Y}_{r,(n-1,1)}$-modules appearing in the restriction of
a simple $\widehat{Y}_{r,n}$-module to the subalgebra $\widehat{Y}_{r,(n-1,1)}$ as well as their multiplicities, as an analog of the case of symmetric groups and affine Hecke algebras of type $A$ (see \cite{Kle1} and \cite{GV}). We remark that due to the appearance of the cyclic group $G$, it is natural to consider the restriction to the subalgebra $\widehat{Y}_{r,(n-1,1)}$ rather than $\widehat{Y}_{r,n-1}$ as we shall see that the multiplicity-free property holds.

For each $a\in \mathbb{K}^{*}$ and $M\in \widehat{\mathcal{H}}_{m}$-$\mathbf{mod},$ let $\Delta_{a}(M)$ be the \emph{generalized} $a$-\emph{eigenspace} of $Y_{m}$ in $\mathrm{Res}_{\widehat{\mathcal{H}}_{m-1,1}}^{\widehat{\mathcal{H}}_{m}}M,$ where $\widehat{\mathcal{H}}_{m-1,1}$ is the subalgebra of $\widehat{\mathcal{H}}_{m}$ generated by $T_{i},$ $Y_{j}^{\pm 1}$ (with $1\leq i\leq m-2$ and $1\leq j\leq m$). Since $Y_{m}-a$ is a central element in $\widehat{\mathcal{H}}_{m-1,1},$ $\Delta_{a}(M)$ is an $\widehat{\mathcal{H}}_{m-1,1}$-submodule of $\mathrm{Res}_{\widehat{\mathcal{H}}_{m-1,1}}^{\widehat{\mathcal{H}}_{m}}M.$ We set $$e_{a}M :=\mathrm{Res}_{\widehat{\mathcal{H}}_{m-1}}^{\widehat{\mathcal{H}}_{m-1,1}}\Delta_{a}(M).$$
Then we have $$\mathrm{Res}_{\widehat{\mathcal{H}}_{m-1}}^{\widehat{\mathcal{H}}_{m}}M=\bigoplus_{a\in \mathbb{K}^{*}}e_{a}M.$$
We denote the \emph{socle} of the $\widehat{\mathcal{H}}_{m-1}$-module $e_{a}M$ by $$\tilde{e}_{a}M :=\mathrm{Soc}(e_{a}M).$$

The following modular branching rule for $\widehat{\mathcal{H}}_{m}$ is due to Grojnowski and Vazirani.
\begin{proposition} {\rm (See \cite[Theorems (A) and (B)]{GV}.)}\label{GV}
Assume that $M$ is a simple $\widehat{\mathcal{H}}_{m}$-module and $a\in \mathbb{K}^{*}.$ Then either $\tilde{e}_{a}M=0$ or $\tilde{e}_{a}M$ is simple. Moreover, the socle of $\mathrm{Res}_{\widehat{\mathcal{H}}_{m-1}}^{\widehat{\mathcal{H}}_{m}}M$ is multiplicity free.
\end{proposition}

We state the following result, which will be used in the proof of Lemma \ref{branch-rule-2}.
\begin{lemma}\label{branch-rule-1}
Suppose that $\mu=(\mu_{1},\ldots,\mu_{r})\in \mathcal{C}_{r}(n)$ and  let $L_{k}$ $(1\leq k\leq r)$ be a simple $\widehat{\mathcal{H}}_{\mu_{k}}$-module. Then for each $\tau\in \mathfrak{S}_{r},$ we have
\begin{equation}\label{isomorphism-add-add}
\mathrm{Ind}_{\widehat{Y}_{r,\mu}}^{\widehat{Y}_{r,n}}\big(V(\mu)\otimes (L_{1}\otimes\cdots\otimes L_{r})\big)
\cong\mathrm{Ind}_{\widehat{Y}_{r,\tau(\mu)}}^{\widehat{Y}_{r,n}}\big(V_{\tau(1)}^{\otimes \mu_{\tau(1)}}\otimes\cdots\otimes V_{\tau(r)}^{\otimes \mu_{\tau(r)}}\otimes (L_{\tau(1)}\otimes\cdots\otimes L_{\tau(r)})\big),
\end{equation}
where $\tau(\mu)=(\mu_{\tau(1)},\ldots,\mu_{\tau(r)}).$
\end{lemma}
\begin{proof}
We denote the left-hand side and the right-hand side of \eqref{isomorphism-add-add} by $\mathrm{L}$ and $\mathrm{R},$ respectively. In order to prove that $\mathrm{L}\cong \mathrm{R},$ it suffices to show that $\mathcal{F}(\mathrm{L})\cong \mathcal{F}(\mathrm{R})$ by Theorem \ref{weig-e-6}. On the one hand, for any $\nu\in \mathcal{C}_{r}(n)$ with $\nu\neq\mu,$ we have $I_{\nu}\mathrm{L}=I_{\nu}\mathrm{R}=0,$ and hence $\mathrm{Hom}_{\mathbb{K}T}(V(\nu), I_{\nu}\mathrm{L})=\mathrm{Hom}_{\mathbb{K}T}(V(\nu), I_{\nu}\mathrm{R})=0.$ On the other hand, by Lemma \ref{weig-e-1} and Proposition \ref{weig-e-4} we have the following isomorphisms: $$\mathrm{Hom}_{\mathbb{K}T}(V(\mu), I_{\mu}\mathrm{L})\cong L_{1}\otimes\cdots\otimes L_{r}\cong\mathrm{Hom}_{\mathbb{K}T}(V(\mu), I_{\mu}\mathrm{R}).$$ Thus, we have proved this lemma.
\end{proof}

Assume that $\mu=(\mu_{1},\ldots,\mu_{r})\in \mathcal{C}_{r}(n).$ For each $1\leq i\leq r,$ we denote by $$\mu_{i}^{-}=(\mu_{1},\ldots,\mu_{i}-1,\ldots,\mu_{r})\quad\text{and}\quad\mu_{i}^{+}=(\mu_{1},\ldots,\mu_{i}+1,\ldots,\mu_{r})$$ the $r$-compositions of $n\mp1$ associated with $\mu,$ respectively. In the following, we shall assume that the terms involving $\mu_{i}^{-}$ are zero if $\mu_{i}=0$ for some $i.$

Recall that $\widehat{Y}_{r,(n-1,1)}$ is the subalgebra of $\widehat{Y}_{r,n}$ generated by $t_1,\ldots,t_{n},$ $X_{1}^{\pm1},\ldots,X_{n}^{\pm1}$ and $g_{w}$ for all $w\in \mathfrak{S}_{n-1}$. Then we have $\widehat{Y}_{r,(n-1,1)}\cong \widehat{Y}_{r,n-1}\otimes\widehat{Y}_{r,1},$ and we shall identify the two algebras in the following.


\begin{lemma}\label{branch-rule-2}
Suppose that $\mu=(\mu_{1},\ldots,\mu_{r})\in \mathcal{C}_{r}(n)$ and $L_{k}$ $(1\leq k\leq r)$ is a simple $\widehat{\mathcal{H}}_{\mu_{k}}$-module. Then we have
$$\mathrm{Res}_{\widehat{Y}_{r,(n-1,1)}}^{\widehat{Y}_{r,n}}S_{\mu}(L.)\cong \bigoplus_{a\in \mathbb{K}^{*}; 1\leq k\leq r} S_{\mu_{k}^{-}}(e_{a}L.)\otimes (V_{k}\otimes L(a)),$$
where $L(a)$ is the one-dimensional $\mathbb{K}[X^{\pm1}]$-module with $X^{\pm1}$ acting as the scalar $a^{\pm1}$ and $S_{\mu_{k}^{-}}(e_{a}L.)$ denotes the $\widehat{Y}_{r,n-1}$-module
\begin{equation}\label{Smuk-}
\mathrm{Ind}_{\widehat{Y}_{r,\mu_{k}^{-}}}^{\widehat{Y}_{r,n-1}}\big(V(\mu_{k}^{-})\otimes (L_{1}\otimes\cdots\otimes e_{a}L_{k}\otimes\cdots\otimes L_{r})\big).
\end{equation}
\end{lemma}
\begin{proof}
For each $1\leq k\leq r$ such that $\mu_k\neq 0$, there exists $\tau\in\mathfrak{S}_r$ such that $\tau(i)=i$ for $1\leq i\leq k-1,$ $\tau(k+j)=k+j+1$ for $0\leq j\leq r-k-1$ and $\tau(r)=k.$ Then $\tau(\mu)=(\mu_1,\ldots,\mu_{k-1},\mu_{k+1},\ldots,\mu_r,\mu_k).$ By Lemma \ref{branch-rule-1} we have
\begin{align*}
S_{\mu}(L.)\cong & \mathrm{Ind}_{\widehat{Y}_{r,\tau(\mu)}}^{\widehat{Y}_{r,n}}\big((V_{\tau(1)}^{\otimes \mu_{\tau(1)}}\otimes\cdots\otimes V_{\tau(r)}^{\otimes \mu_{\tau(r)}})\otimes (L_{\tau(1)}\otimes\cdots\otimes L_{\tau(r)})\big)\\
= & \mathrm{Ind}_{\widehat{Y}_{r,\tau(\mu)}}^{\widehat{Y}_{r,n}}\big((V_{\tau(1)}^{\otimes \mu_{\tau(1)}}\otimes\cdots\otimes V_{\tau(r-1)}^{\otimes \mu_{\tau(r-1)}}\otimes V_k^{\otimes \mu_k})\otimes (L_{\tau(1)}\otimes\cdots\otimes L_{\tau(r-1)}\otimes L_k)\big)
\end{align*}
By definition, it is easy to see that
$$\bigg(\hspace{-0.5mm}\mathrm{Ind}_{\widehat{Y}_{r,\tau(\mu)^-}}^{\widehat{Y}_{r,n-1}}\hspace{-1mm}\big(V_{\tau(1)}^{\otimes \mu_{\tau(1)}}\hspace{-1mm}\otimes\cdots\otimes V_{\tau(r-1)}^{\otimes \mu_{\tau(r-1)}}\hspace{-1mm}\otimes V_k^{\otimes\mu_k-1}\otimes (L_{\tau(1)}\otimes\cdots\otimes L_{\tau(r-1)}\otimes e_aL_k)\big)\hspace{-1mm}\bigg)\otimes\big(V_k\otimes L(a))$$
is isomorphic to a $\widehat{Y}_{r,(n-1,1)}$-submodule of $\mathrm{Res}_{\widehat{Y}_{r,(n-1,1)}}^{\widehat{Y}_{r,n}}S_{\mu}(L.)$ for all $a\in \mathbb{K}^{*},$ where $\tau(\mu)^-=(\mu_1,\ldots,\mu_{k-1},\mu_{k+1},\ldots,\mu_r,\mu_k-1).$ Meanwhile by a proof similar to Lemma \ref{branch-rule-1} we can show
\begin{align*}
\mathrm{Ind}_{\widehat{Y}_{r,\tau(\mu)^-}}^{\widehat{Y}_{r,n}}\hspace{-1.5mm}\big(V_{\tau(1)}^{\otimes \mu_{\tau(1)}}\hspace{-1mm}\otimes\cdots\otimes V_{\tau(r-1)}^{\otimes \mu_{\tau(r-1)}}\hspace{-1mm}\otimes V_k^{\otimes\mu_k-1}\hspace{-1mm}\otimes (L_{\tau(1)}\hspace{-0.5mm}\otimes\cdots\otimes \hspace{-0.5mm}L_{\tau(r-1)}\hspace{-0.5mm}\otimes\hspace{-0.5mm} e_aL_k)\big)
\hspace{-0.5mm}\cong\hspace{-0.5mm} S_{\mu_{k}^{-}}(e_{a}L.)
\end{align*}
Putting together, we obtain that $S_{\mu_{k}^{-}}(e_{a}L.)\otimes (V_{k}\otimes L(a))$ is a $\widehat{Y}_{r,(n-1,1)}$-submodule of $\mathrm{Res}_{\widehat{Y}_{r,(n-1,1)}}^{\widehat{Y}_{r,n}}S_{\mu}(L.)$ for each $a\in \mathbb{K}^{*}$ and $1\leq k\leq r,$ and hence we have
\begin{equation}\label{inclusion-yh-algs}
\sum_{a\in \mathbb{K}^{*}; 1\leq k\leq r}S_{\mu_{k}^{-}}(e_{a}L.)\otimes (V_{k}\otimes L(a))\subseteq \mathrm{Res}_{\widehat{Y}_{r,(n-1,1)}}^{\widehat{Y}_{r,n}}S_{\mu}(L.).\end{equation}
Since $V_{k}\otimes L(a)$ are pairwise non-isomorphic simple $\widehat{Y}_{r,1}$-modules for distinct pairs $(k,a),$ the above sum is in fact a direct sum. This lemma now follows from a dimension comparison. In fact, the dimensions of two sides of \eqref{inclusion-yh-algs} are both $$\dim{V(\mu)}\cdot \prod_{k=1}^{r}\dim{L_{k}}\cdot \frac{n!}{\mu_{1}!\mu_{2}!\cdots \mu_{r}!}.$$
\end{proof}

Now we can establish the modular branching rules for $\widehat{Y}_{r,n}.$
\begin{theorem}\label{branch-rule-3}
Consider the simple $\widehat{Y}_{r,n}$-module $S_{\mu}(L.)$ defined in Theorem \ref{simple-mod-1}. Then we have $$\mathrm{Soc}(\mathrm{Res}_{\widehat{Y}_{r,(n-1,1)}}^{\widehat{Y}_{r,n}}S_{\mu}(L.))\cong \bigoplus_{a\in \mathbb{K}^{*}; 1\leq k\leq r} S_{\mu_{k}^{-}}(\tilde{e}_{a}L.)\otimes (V_{k}\otimes L(a)),$$
where $S_{\mu_{k}^{-}}(\tilde{e}_{a}L.)$ denotes the $\widehat{Y}_{r,n-1}$-module
$$\mathrm{Ind}_{\widehat{Y}_{r,\mu_{k}^{-}}}^{\widehat{Y}_{r,n-1}}\big(V(\mu_{k}^{-})\otimes (L_{1}\otimes\cdots\otimes \tilde{e}_{a}L_{k}\otimes\cdots\otimes L_{r})\big).$$
\end{theorem}
\begin{proof}
For $1\leq k\leq r,$ recall that $S_{\mu_{k}^{-}}(e_{a}L.)$ has been defined via \eqref{Smuk-}. Suppose that $M$ is a simple $\widehat{Y}_{r,n-1}$-submodule of $S_{\mu_{k}^{-}}(e_{a}L.)$. By Theorem \ref{simple-mod-1}, we can assume
$M=S_{\lambda}(J.)=\mathrm{Ind}_{\widehat{Y}_{r,\lambda}}^{\widehat{Y}_{r,n-1}}\big(V(\lambda)\otimes (J_{1}\otimes\cdots\otimes J_{r})\big)$ with $\lambda=(\lambda_1,\ldots,\lambda_r)$ being a composition of $n-1$ and $J_k$ $(1\leq k\leq r)$ being a simple $\widehat{\mathcal{H}}_{\lambda_{k}}$-module. Then by Frobenius reciprocity (see \cite[Chapter II, $\S5.1$]{Bo}), there exists a nonzero $\widehat{Y}_{r,\lambda}$-homomorphism
from $V(\lambda)\otimes (J_{1}\otimes\cdots\otimes J_{r})$ to $\mathrm{Res}^{\widehat{Y}_{r,n-1}}_{\widehat{Y}_{r,\lambda}}S_{\mu_{k}^{-}}(e_{a}L.).$ By considering the decompositions of them into direct sum of simple modules as $\mathbb{K}T$-modules,
we can deduce that $\lambda=\mu_k^-$ and $V(\lambda)\otimes (J_{1}\otimes\cdots\otimes J_{r})$ is isomorphic to a $\widehat{Y}_{r,\lambda}$-submodule of
$V(\lambda)\otimes (L_1\otimes\cdots\otimes e_aL_k\otimes\cdots\otimes L_r).$ Hence $J_i\cong L_i$ for $i\neq k$ and $J_k$
is isomorphic to an $\widehat{H}_{\mu_k-1}$-submodule of $e_aL_k$. We must have $J_k\cong \tilde{e}_aL_k$ by Proposition \ref{GV}. Therefore, the socle of the $\widehat{Y}_{r,n-1}$-module $S_{\mu_{k}^{-}}(e_{a}L.)$ is $S_{\mu_{k}^{-}}(\tilde{e}_{a}L.).$ The theorem follows from Lemma \ref{branch-rule-2}.
\end{proof}

\subsection{A block decomposition}
In this subsection we fix a module $M$ in $\widehat{Y}_{r,n}$-$\mathbf{mod}.$ We shall give a decomposition of $M$ following the approach in \cite[Sections 4.1 and 4.2]{Kle2}.

For any $\underline{s}=(s_1,\ldots,s_{n})\in (\mathbb{K}^{*})^{n},$ let $M_{\underline{s}}$ be the simultaneous generalized eigenspace of the commutative invertible elements $X_{1},\ldots,X_{n}$ acting on $M$ with eigenvalues $s_1,\ldots,s_{n}.$ Then as a $P_{n}$-module, we have $$M=\bigoplus_{\underline{s}\in (\mathbb{K}^{*})^{n}}M_{\underline{s}}.$$


Set $\Lambda_{n} :=\mathbb{K}[X_{1}^{\pm1},\ldots,X_{n}^{\pm1}]^{\mathfrak{S}_{n}}.$ Associated with each $\underline{s}\in (\mathbb{K}^{*})^{n},$ we define a one-dimensional representation of $\Lambda_{n}$ by
$$\omega_{\underline{s}}: \Lambda_{n}\rightarrow \mathbb{K},\quad f(X_{1}^{\pm1},\ldots,X_{n}^{\pm1})\mapsto f(s_{1}^{\pm1},\ldots,s_{n}^{\pm1}).$$
If $\underline{s}$ and $\underline{t}$ lie in the same $\mathfrak{S}_{n}$-orbit, we write $\underline{s}\sim \underline{t}$. Note that $\underline{s}\sim \underline{t}$ if and only if $\omega_{\underline{s}}=\omega_{\underline{t}}.$ For each orbit $\gamma\in (\mathbb{K}^{*})^{n}/\sim,$ we set $\omega_{\gamma} :=\omega_{\underline{s}}$ for any $\underline{s}\in \gamma.$ Then $\omega_{\gamma}$ is well-defined. Set $$M[\gamma] :=\big\{m\in M\:|\:(z-\omega_{\gamma}(z))^{N}m=0~\mathrm{for~all~}z\in \Lambda_{n}~\mathrm{and}~N\gg 0\big\}.$$
Then we have $$M[\gamma]=\bigoplus_{\underline{s}\in \gamma}M_{\underline{s}}.$$

Since $\Lambda_{n}$ lies in the center of $\widehat{Y}_{r,n}$ by Theorem \ref{center}, $M[\gamma]$ is a $\widehat{Y}_{r,n}$-module. Moreover, we have the following decomposition in $\widehat{Y}_{r,n}$-$\mathbf{mod}$:
\begin{equation}\label{mgamma}
M=\bigoplus_{\gamma\in (\mathbb{K}^{*})^{n}/\sim}M[\gamma].
\end{equation}
Recall the decomposition in Lemma \ref{weig-e-2}. For each $\mu\in \mathcal{C}_{r}(n)$ and $\gamma\in (\mathbb{K}^{*})^{n}/\sim,$ we set $M[\mu, \gamma] :=M_{\mu}\cap M[\gamma].$ Since $X_{1}^{\pm1},\ldots,X_{n}^{\pm1}$ commute with $t_1,\ldots,t_{n},$ we have $M[\mu, \gamma]=(M_{\mu})[\gamma]=(M[\gamma])_{\mu}.$ Combining this with Lemma \ref{weig-e-2} and (\ref{mgamma}), we get the following decomposition in $\widehat{Y}_{r,n}$-$\mathbf{mod}$:
\begin{equation}\label{mgammamu}
M=\bigoplus_{\mu\in \mathcal{C}_{r}(n); \gamma\in (\mathbb{K}^{*})^{n}/\sim}M[\mu, \gamma].
\end{equation}

The decomposition \eqref{mgammamu} provides us a block decomposition of $\widehat{Y}_{r,n}$-$\mathbf{mod}$ by applying Theorem \ref{weig-e-6} and the block decomposition for the extended affine Hecke algebra $\widehat{\mathcal{H}}_{m}$ over an algebraically closed field; see \cite[Proposition 4.4]{Gr} and also \cite[Theorem 2.15]{LM}.

\section{Cyclotomic Yokonuma-Hecke algebras and Morita equivalences}
In this section, for a cyclotomic Yokonuma-Hecke algebra $Y_{r,n}^{\lambda}$ (see \eqref{cyclotomic-yhc-alge-defini}), we establish an explicit equivalence between the category $Y_{r,n}^{\lambda}$-$\mathbf{mod}$ of finite dimensional $Y_{r,n}^{\lambda}$-modules and the category $\mathcal{H}_{r,n}^{\lambda}$-$\mathbf{mod}$ of finite dimensional $\mathcal{H}_{r,n}^{\lambda}$-modules, where $\mathcal{H}_{r,n}^{\lambda}$ is a direct sum of tensor products of various Ariki-Koike algebras (see \eqref{hrn-hlambda-add-yh-dda}). The category equivalence plays a crucial role in Section 6.


\subsection{Cyclotomic Yokonuma-Hecke algebras}
Set $\mathbb{I} :=\{q^{i}\:|\:i\in \mathbb{Z}\}.$ For a $\widehat{Y}_{r,n}$-module $M,$ we call $M$ \emph{integral} if it is finite dimensional and all eigenvalues of $X_{1},\ldots,X_{n}$ acting on it belong to the set $\mathbb{I}.$ We denote by $\widehat{Y}_{r,n}$-$\mathbf{mod}_{\mathbb{I}}$ the full subcategory of $\widehat{Y}_{r,n}$-$\mathbf{mod}$ consisting of all integral $\widehat{Y}_{r,n}$-modules. Similarly, we can define integral $\widehat{\mathcal{H}}_{n}$-modules and its subcategory $\widehat{\mathcal{H}}_{n}$-$\mathbf{mod}_{\mathbb{I}}.$ It is explained in \cite[Remark 1]{Va} that to understand $\widehat{\mathcal{H}}_{n}$-$\mathbf{mod}$, it suffices to understand $\widehat{\mathcal{H}}_{n}$-$\mathbf{mod}_{\mathbb{I}},$ that is, the study of simple modules of $\widehat{\mathcal{H}}_{n}$ can be reduced to that of integral simple $\widehat{\mathcal{H}}_{n}$-modules. Then by Theorem \ref{weig-e-6}, it suffices to study simple objects in $\widehat{Y}_{r,n}$-$\mathbf{mod}_{\mathbb{I}}$ in order to study simple $\widehat{Y}_{r,n}$-modules.


Now we introduce the following intertwining elements in $\widehat{Y}_{r,n}$:
$$\Theta_{i} :=g_{i}(1-X_{i}X_{i+1}^{-1})+(1-q)e_{i}\quad \text{for }1\leq i\leq n-1.$$
\begin{lemma}\label{intertwining} For each $1\leq i\leq n-1,$ we have
\begin{equation}\label{theta-i-2}
\Theta_{i}^{2}=(1-q)^{2}(e_{i}-1)+(1-qX_{i}X_{i+1}^{-1})(1-qX_{i+1}X_{i}^{-1});\end{equation}
\begin{equation}\label{theta-i-x}
\Theta_{i}X_{i}=X_{i+1}\Theta_{i},\quad \Theta_{i}X_{i+1}=X_{i}\Theta_{i}\quad\text{and}\quad\Theta_{i}X_{j}=X_{j}\Theta_{i}~~\text{for}~j\neq i,i+1.
\end{equation}
\end{lemma}
\begin{proof}
By (\ref{gxxg}), we can prove these identities by a direct computation.
\begin{align*}
\Theta_{i}^{2}=&\big(g_{i}(1-X_{i}X_{i+1}^{-1})+(1-q)e_{i}\big)^{2}\\
=&g_{i}(1-X_{i}X_{i+1}^{-1})g_{i}(1-X_{i}X_{i+1}^{-1})+2(1-q)g_{i}e_{i}(1-X_{i}X_{i+1}^{-1})+(1-q)^{2}e_{i}^{2}\\
=&(q+(q-1)e_{i}g_{i})(1-X_{i}X_{i+1}^{-1})-g_{i}X_{i}(g_{i}X_{i}^{-1}-(q-1)e_{i}X_{i}^{-1})\\
&\times(1-X_{i}X_{i+1}^{-1})+2(1-q)g_{i}e_{i}(1-X_{i}X_{i+1}^{-1})+(1-q)^{2}e_{i}\\
=&q(1-X_{i}X_{i+1}^{-1})+(q-1)g_{i}e_{i}(1-X_{i}X_{i+1}^{-1})-qX_{i+1}X_{i}^{-1}(1-X_{i}X_{i+1}^{-1})\\
&+(q-1)g_{i}e_{i}(1-X_{i}X_{i+1}^{-1})+2(1-q)g_{i}e_{i}(1-X_{i}X_{i+1}^{-1})+(1-q)^{2}e_{i}\\
=&(1-q)^{2}(e_{i}-1)+(1-qX_{i}X_{i+1}^{-1})(1-qX_{i+1}X_{i}^{-1}).
\end{align*}
\begin{align*}
\Theta_{i}X_{i}&=\big(g_{i}(1-X_{i}X_{i+1}^{-1})+(1-q)e_{i}\big)X_{i}\\
&=(X_{i+1}g_{i}-(q-1)e_{i}X_{i+1})(1-X_{i}X_{i+1}^{-1})+(1-q)e_{i}X_{i}\\
&=X_{i+1}g_{i}(1-X_{i}X_{i+1}^{-1})-(q-1)e_{i}X_{i+1}+(q-1)e_{i}X_{i}+(1-q)e_{i}X_{i}\\
&=X_{i+1}\big(g_{i}(1-X_{i}X_{i+1}^{-1})+(1-q)e_{i}\big)\\
&=X_{i+1}\Theta_{i}.
\end{align*}
\begin{align*}
\Theta_{i}X_{i+1}&=\big(g_{i}(1-X_{i}X_{i+1}^{-1})+(1-q)e_{i}\big)X_{i+1}\\
&=(X_{i}g_{i}+(q-1)e_{i}X_{i+1})(1-X_{i}X_{i+1}^{-1})+(1-q)e_{i}X_{i+1}\\
&=X_{i}g_{i}(1-X_{i}X_{i+1}^{-1})+(q-1)e_{i}X_{i+1}-(q-1)e_{i}X_{i}+(1-q)e_{i}X_{i+1}\\
&=X_{i}\big(g_{i}(1-X_{i}X_{i+1}^{-1})+(1-q)e_{i}\big)\\
&=X_{i}\Theta_{i}.
\end{align*}
By (\ref{giXj}) and (\ref{xyyx}), we see that $\Theta_{i}X_{j}=X_{j}\Theta_{i}~\mathrm{for}~j\neq i,i+1.$
\end{proof}
\begin{lemma}\label{integral-2}
Fix $i$ with $1\leq i\leq n$ and let $M\in \widehat{Y}_{r,n}$-$\mathbf{mod}.$ Assume that all eigenvalues of $X_{i}$ on $M$ belong to $\mathbb{I}.$ Then $M$ is integral.
\end{lemma}
\begin{proof}
By assumption, we would like to show that the eigenvalues of $X_{k}$ on $M$ belong to $\mathbb{I}$ if and only if the eigenvalues of $X_{k+1}$ on $M$ belong to $\mathbb{I}$ for $1\leq k\leq n-1.$ By Lemmas \ref{weig-e-1} and \ref{weig-e-2}, it suffices to consider the subspaces $I_{\mu}M$ for all $\mu\in \mathcal{C}_{r}(n).$ Suppose that all eigenvalues of $X_{k+1}$ on $I_{\mu}M$ belong to $\mathbb{I}.$ Let $a$ be an eigenvalue of $X_{k}$ acting on $I_{\mu}M.$ Since $X_{k}$ and $X_{k+1}$ commute, we can choose an element $u$ from the $a$-eigenspace of $X_{k}$ so that $u$ is also an eigenvector of $X_{k+1}$ with an eigenvalue $b.$ By assumption, we have $b=q^{s}$ for some $s\in \mathbb{Z}.$


If $\Theta_{k}u\neq 0,$ then by (\ref{theta-i-x}) we have $X_{k+1}\Theta_{k}u=\Theta_{k}X_{k}u=a\Theta_{k}u.$ Thus, we see that $a$ is an eigenvalue of $X_{k+1},$ and hence $a\in \mathbb{I}$ by assumption. If $\Theta_{k}u=0,$ then by (\ref{theta-i-2}), we have $$(1-q)^{2}(e_{k}-1)u+(1-q^{1-s}a)(1-q^{1+s}a^{-1})u=0.$$
Since $I_{\mu}M$ is isomorphic to the direct sum of finite copies of $V_{1}^{\otimes \mu_{1}}\otimes\cdots\otimes V_{r}^{\mu_{r}},$ we have either $e_{k}u=0$ or $e_{k}u=u$ by Lemma \ref{weig-e-0}. Thus, we must have either $a=q^{s}$ or $a=q^{s\pm1},$ and hence $a\in \mathbb{I}$ again. It is similar to show that all eigenvalues of $X_{k+1}$ on $I_{\mu}M$ belong to $\mathbb{I}$ if we assume all eigenvalues of $X_{k}$ on $I_{\mu}M$ belong to $\mathbb{I}.$
\end{proof}

Recall that $e$ is the order of $q$ in $\mathbb{K}^{*}.$ Then $e\in \mathbb{Z}_{\geq 2}\cup \{\infty\}.$ If $e\in \mathbb{Z}_{\geq 2}$, we set $\mathbb{J}=\{0,1,\ldots,e-1\};$ otherwise, we set $\mathbb{J}=\mathbb{Z}.$
Let $$\Delta=\big\{\lambda=(\lambda_{i})_{i\in \mathbb{J}}\:|\:\lambda_{i}\in \mathbb{Z}_{\geq 0}~\mathrm{and~only~finitely~many}~\lambda_{i}~\mathrm{are~nonzero}\big\}.$$

For each $\lambda\in \Delta,$ set $$f_{\lambda}\equiv f_{\lambda}(X_{1}) :=\prod_{i\in \mathbb{J}}(X_{1}-q^{i})^{\lambda_{i}}.$$

For each $\lambda\in \Delta,$ we denote by $\mathcal{J}_{\lambda}$ the two-sided ideal of $\widehat{Y}_{r,n}$ generated by $f_{\lambda},$ and define the \emph{cyclotomic Yokonuma-Hecke algebra} $Y_{r,n}^{\lambda}$ by
\begin{equation}\label{cyclotomic-yhc-alge-defini}
Y_{r,n}^{\lambda} :=\widehat{Y}_{r,n}/\mathcal{J}_{\lambda}.
\end{equation}
\begin{lemma}\label{integral-3}
Assume that $M\in \widehat{Y}_{r,n}$-$\mathbf{mod}.$ Then $M$ is integral if and only if $\mathcal{J}_{\lambda}M=0$ for some $\lambda\in \Delta.$
\end{lemma}
\begin{proof}
If $\mathcal{J}_{\lambda}M=0,$ then all the eigenvalues of $X_{1}$ on $M$ belong to $\mathbb{I}.$ Hence $M$ is integral by Lemma \ref{integral-2}. Conversely, suppose that $M$ is integral. Then the minimal polynomial of $X_{1}$ on $M$ should be of the form $\prod_{i\in \mathbb{J}}(t-q^{i})^{\lambda_{i}}$ for some $\lambda_{i}\in \mathbb{Z}_{\geq 0}.$ Setting $\mathcal{J}_{\lambda}$ to be the two-sided ideal of $\widehat{Y}_{r,n}$ generated by $\prod_{i\in \mathbb{J}}(X_{1}-q^{i})^{\lambda_{i}},$ we have $\mathcal{J}_{\lambda}M=0.$
\end{proof}

For each $\lambda\in \Delta,$ we have a canonical surjective homomorphism $\widehat{Y}_{r,n}\rightarrow Y_{r,n}^{\lambda},$ via which we can identify $Y_{r,n}^{\lambda}$-$\mathbf{mod}$ with the full subcategory of $\widehat{Y}_{r,n}$-$\mathbf{mod}$ consisting of all modules $M$ with $\mathcal{J}_{\lambda}M=0.$ In order to study modules in the category $\widehat{Y}_{r,n}$-$\mathbf{mod}_{\mathbb{I}},$ it suffices to study modules in the category $Y_{r,n}^{\lambda}$-$\mathbf{mod}$ for all $\lambda\in \Delta$ by Lemma \ref{integral-3}.

The next proposition follows from \cite[Theorem 4.4]{ChP2}. In fact, we can adjust the statements in \cite[Section 7.5]{Kle2} to our setting and give a direct proof of the PBW basis theorem for $Y_{r,n}^{\lambda}$; see \cite[Section 2]{C} for more details. For each $\lambda=(\lambda_{i})_{i\in \mathbb{J}}\in \Delta,$ we set $|\lambda| :=\sum_{i\in \mathbb{J}}\lambda_{i}.$
\begin{proposition}\label{cyclotomic-pbw}
Suppose that $\lambda\in \Delta$. Then the following elements $$\big\{X^{\alpha}t^\beta g_{w}\:|\:\alpha=(\alpha_{1},\ldots,\alpha_{n})\in \mathbb{Z}^{n}~\mathrm{with}~0\leq\alpha_{1},\ldots,\alpha_{n}\leq |\lambda|-1, \beta\in \mathbb{Z}_r^n\text{ and }w\in \mathfrak{S}_{n}\big\}$$
form a $\mathbb{K}$-basis of $Y_{r,n}^{\lambda}.$
\end{proposition}

\subsection{A Morita equivalence}

Let $\mathfrak{S}_{n-1}'$ be the subgroup of $\mathfrak{S}_{n}$ generated by $s_{2},\ldots,s_{n-1}.$ For each $\mu=(\mu_{1},\ldots,\mu_{r})\in \mathcal{C}_{r}(n)$ and $0\leq k\leq r,$ we set $\bar{\mu}^{k} :=\mu_{1}+\cdots+\mu_{k},$ where $\bar{\mu}^{k}=0$ if $k=0.$ The next lemma follows from \cite[Proposition A.3.2]{Ze}.
\begin{lemma}\label{left-coset}{\rm (See \cite[Lemma 5.10]{WW}.)}
There exists a complete set $\mathcal{O}(\mu)$ of left coset representatives of $\mathfrak{S}_{\mu}$ in $\mathfrak{S}_{n}$ such that any $w\in \mathcal{O}(\mu)$ can be written as $\sigma(1,\bar{\mu}^{k}+1)$ for some $\sigma\in \mathfrak{S}_{n-1}'$ and $0\leq k\leq r-1.$ $($Here $(1,\bar{\mu}^{k}+1)=\mathrm{Id}$ if $\bar{\mu}^{k}=0.$$)$
\end{lemma}

Note that $(1, m+1)=s_{m}\cdots s_{2}s_{1}s_{2}\cdots s_{m}$ for $0\leq m\leq n-1.$ By (\ref{gif}) and the identity $e_{i,j}g_{j}=g_{j}e_{i,j+1}$ for $1\leq i< j\leq n-1,$ we can easily get the following result.
\begin{lemma}\label{xggx-commutator}
Assume that $\mu\in\mathcal{C}_{r}(n)$. Fix $k$ with $0\leq k\leq r-1$ and set $w^k_\mu :=(1,\bar{\mu}^{k}+1).$ Then we have
$$X_{1}g_{w_\mu^k}=g_{w_\mu^k}X_{\bar{\mu}^{k}+1}-(q-1)\sum_{l=1}^{\bar{\mu}^{k}}g_{\bar{\mu}^{k}}\cdots g_{2}g_{1}g_{2}\cdots \widehat{g}_{l}^{X_{l+1}}\cdots g_{\bar{\mu}^{k}}e_{l, \bar{\mu}^{k}+1},$$
where $\widehat{g}_{l}^{X_{l+1}}$ means replacing $g_{l}$ with $X_{l+1}.$
\end{lemma}

Assume that $\{\alpha_{i}\:|\:i\in \mathbb{J}\}$ is the set of simple roots of the affine Lie algebra $\widehat{sl}_{e}$ and $\{\alpha_{i}^{\vee}\:|\:i\in \mathbb{J}\}$ is the set of the corresponding simple coroots. Let $P_{+}$ be the set of all dominant integral weights of $\widehat{sl}_{e}.$
For each $\mu\in P_{+},$  following \cite{AK} the associated \emph{Ariki-Koike algebra} $\mathcal{H}_{m}^{\mu}$ is defined by $$\mathcal{H}_{m}^{\mu}=\widehat{\mathcal{H}}_{m}\Big/\Big\langle\prod_{i\in \mathbb{J}}(Y_{1}-q^{i})^{\langle \alpha_{i}^{\vee}, \mu\rangle}\Big\rangle.$$

For each $\lambda\in \Delta,$ we define $\lambda'\in P_{+}$ by setting $\langle \alpha_{i}^{\vee}, \lambda'\rangle=\lambda_{i}$ for any $i\in \mathbb{J}.$ Thus, we have a one-to-one correspondence between $\Delta$ and $P_{+},$ and we shall identify the two sets. For each $\lambda\in \Delta,$ we define the following algebra:
\begin{equation}\label{hrn-hlambda-add-yh-dda}
\mathcal{H}_{r,n}^{\lambda}=\bigoplus_{\mu\in \mathcal{C}_{r}(n)} \mathcal{H}_{\mu_{1}}^{\lambda}\otimes\cdots\otimes\mathcal{H}_{\mu_{r}}^{\lambda}.
\end{equation}

Recall the functor $\mathcal{F}$ defined in Theorem \ref{weig-e-6}. Then we have the following result.
\begin{theorem}\label{morita-equi}
Fix one $\lambda\in \Delta.$ Then the functor $\mathcal{F}$ induces an equivalence $\mathcal{F}^{\lambda}$ between the categories $Y_{r,n}^{\lambda}$-$\mathbf{mod}$ and $\mathcal{H}_{r,n}^{\lambda}$-$\mathbf{mod}.$
\end{theorem}
\begin{proof}
Recall that the category $Y_{r,n}^{\lambda}$-$\mathbf{mod}$ can be identified with the full subcategory of $\widehat{Y}_{r,n}$-$\mathbf{mod}$ consisting of all modules which are annihilated by $\mathcal{J}_{\lambda}.$ Assume that $M\in \widehat{Y}_{r,n}$-$\mathbf{mod}.$ By Lemma \ref{weig-e-2}, we see that $\mathcal{J}_{\lambda}M=0$ if and only if $\mathcal{J}_{\lambda}M_{\mu}=0$ for each $\mu\in \mathcal{C}_{r}(n).$ Fix one $\mu\in \mathcal{C}_{r}(n).$ By Lemma \ref{weig-e-1}, we have $M_{\mu}\cong \mathrm{Ind}_{\widehat{Y}_{r,\mu}}^{\widehat{Y}_{r,n}}(I_{\mu}M).$ This together with the fact that $g_w\otimes I_{\mu}M=g_\tau\otimes I_{\mu}M$ for any $w=\tau w_{\mu}$ with $\tau\in \mathcal{O}(\mu)$ and $w_{\mu}\in \mathfrak{S}_{\mu}$ and a dimension comparison implies
\begin{equation}\label{mmubigoplus-add}
M_{\mu}\cong\bigoplus_{w\in \mathcal{O}(\mu)}g_{w}\otimes I_{\mu}M
\end{equation}
as $\mathbb{K}$-vector spaces.

For each $w\in \mathcal{O}(\mu)$, there exists $\sigma\in \mathfrak{S}_{n-1}'$ such that $w=\sigma(1,\bar{\mu}^{k}+1)=\sigma w^k_\mu$ for some $0\leq k\leq r-1$ by Lemma \ref{left-coset}. Note that $e_{l, \bar{\mu}^{k}+1}$ acts as zero on $I_{\mu}M$ for all $1\leq l\leq \bar{\mu}^{k}.$ Then by Lemma \ref{xggx-commutator}, we have $X_{1}g_{w^k_\mu}\otimes z=g_{w^k_\mu}\otimes X_{\bar{\mu}^{k}+1}z$ for any $z\in I_{\mu}M,$ and hence
\begin{equation}\label{flambda-add-dda}
f_{\lambda}g_{w}\otimes z=g_{w}\otimes f_{\lambda, k}z,
\end{equation}
where $f_{\lambda, k} :=\prod_{i\in \mathbb{J}}(X_{\bar{\mu}^{k}+1}-q^{i})^{\lambda_{i}}.$


By \eqref{mmubigoplus-add} and \eqref{flambda-add-dda}, we have $f_{\lambda}M_{\mu}=0$ if and only if $f_{\lambda, k}I_{\mu}M=0$ for all $0\leq k\leq r-1.$ By Proposition \ref{weig-e-5}, we have $I_{\mu}M\cong V(\mu)\otimes \mathrm{Hom}_{\mathbb{K}T}(V(\mu), I_{\mu}M).$ Moreover, by Proposition \ref{weig-e-4}, we see that for all $0\leq k\leq r-1,$ $f_{\lambda, k}$ acts as zero on $I_{\mu}M$ if and only if $\prod_{i\in \mathbb{J}}(Y_{\bar{\mu}^{k}+1}-q^{i})^{\lambda_{i}}$ acts as zero on $\mathrm{Hom}_{\mathbb{K}T}(V(\mu), I_{\mu}M).$

By the above arguments, we have $f_{\lambda}M=0$ if and only if $\mathrm{Hom}_{\mathbb{K}T}(V(\mu), I_{\mu}M)\in \mathcal{H}_{r,n}^{\lambda}$-$\mathbf{mod}$ for each $\mu\in \mathcal{C}_{r}(n).$ We are done.
\end{proof}

\section{Simple $Y_{r,n}^{\lambda}$-modules and modular branching rules for $Y_{r,n}^{\lambda}$}
In this section, we shall present several applications of the category equivalence established in Theorem \ref{morita-equi}. We shall classify all finite dimensional simple $Y_{r,n}^{\lambda}$-modules, and establish the modular branching rule for $Y_{r,n}^{\lambda}$ which gives a description of the socle of the restriction to $Y_{r,(n-1,1)}^{\lambda}$ of a simple $Y_{r,n}^{\lambda}$-module, where $Y_{r,(n-1,1)}^{\lambda}$ is a subalgebra of $Y_{r,n}^{\lambda}.$ We also provide a crystal graph interpretation for the modular branching rule of $Y_{r,n}^{\lambda}.$ In the end, we shall give a block decomposition of $Y_{r,n}^{\lambda}$-$\mathbf{mod}.$


\subsection{Simple $Y_{r,n}^{\lambda}$-modules}
Fix one $\lambda\in \Delta$ and set $d :=|\lambda|.$ For each $m,$ let $\mathrm{ev}_{m,\lambda}$ denote the surjective algebra homomorphism $\mathrm{ev}_{m,\lambda}: \widehat{\mathcal{H}}_{m}\rightarrow\mathcal{H}_{m}^{\lambda}.$ Then an $\mathcal{H}_{m}^{\lambda}$-module $L$ can be regarded as an $\widehat{\mathcal{H}}_{m}$-module by inflation, which we shall denote by $\mathrm{ev}^*_{m,\lambda}L$. From the proof of Theorem \ref{morita-equi}, we see that if $L_{k}$ $(1\leq k\leq r)$ is a simple $\mathcal{H}_{\mu_{k}}^{\lambda}$-module, then $S_{\mu}(L.)$ is in fact a $Y_{r,n}^{\lambda}$-module. Thus, by Theorem \ref{simple-mod-1} we immediately obtain the following result.
\begin{theorem}\label{cycloto-simple-mod-1}
Each simple $Y_{r,n}^{\lambda}$-module is isomorphic to a module of the form
$$S_{\mu}(L.) :=\mathrm{Ind}_{\widehat{Y}_{r,\mu}}^{\widehat{Y}_{r,n}}\big(V(\mu)\otimes (\mathrm{ev}_{\mu_1,\lambda}^{*}L_{1}\otimes\cdots\otimes \mathrm{ev}_{\mu_r,\lambda}^{*}L_{r})\big),$$
where $\mu=(\mu_{1},\ldots,\mu_{r})\in \mathcal{C}_{r}(n)$ and $L_{k}$ $(1\leq k\leq r)$ is a simple $\mathcal{H}_{\mu_{k}}^{\lambda}$-module. Moreover, the above modules $S_{\mu}(L.),$ with $\mu=(\mu_{1},\ldots,\mu_{r})$ running through $\mathcal{C}_{r}(n)$ and $L_{k}$ $(1\leq k\leq r)$ running through all non-isomorphic simple $\mathcal{H}_{\mu_{k}}^{\lambda}$-modules, form a complete set of pairwise non-isomorphic simple $Y_{r,n}^{\lambda}$-modules.
\end{theorem}

Recall that the classification of simple modules of Ariki-Koike algebras over an arbitrary field has been given by Ariki in terms of Kleshchev multipartitions. Let $\mathcal{I}_{m}^{\lambda}$ be the set of all $d$-multipartitions of $m.$ We denote by $\mathcal{K}_{m}^{\lambda}$ the set of all Kleshchev multipartitions in $\mathcal{I}_{m}^{\lambda};$ see \cite[Definition 2.3]{Ari2} for the precise definition. Then the simple $\mathcal{H}_{m}^{\lambda}$-modules are parameterized by $\mathcal{K}_{m}^{\lambda};$ see \cite[Theorem 4.2]{Ari2}. Combining this with Theorem \ref{cycloto-simple-mod-1}, we immediately obtain the next result.
\begin{corollary}
The simple $Y_{r,n}^{\lambda}$-modules are parameterized by the following set $$\mathcal{B} :=\big\{(\mu, \psi_{1},\ldots,\psi_{r})\:|\:\mu=(\mu_{1},\ldots,\mu_{r})\in \mathcal{C}_{r}(n)\text{ and }\psi_{i} \in \mathcal{K}_{\mu_{i}}^{\lambda}\text{ for }1\leq i\leq r\big\}.$$
\end{corollary}

\begin{remark} The simple modules of a cyclotomic Yokonuma-Hecke algebra in the generic semisimple case have been classified in \cite[Proposition 3.4]{ChP2}.\end{remark}

In the case that $d=1,$ $Y_{r,n}^{\lambda}$ is the Yokonuma-Hecke algebra $Y_{r,n}$ (see Remark \ref{rem-YH}), and $\mathcal{K}_{m}^{\lambda}$ is exactly the set of $e$-restricted partitions of $m$ (recall that $e$ is the order of $q$ in $\mathbb{K}^{*}$). Thus, we recover the following result due to Jacon and Poulain d'Andecy.
\begin{corollary}{\rm (See \cite[Section 4.1]{JPA}.)}
The simple $Y_{r,n}$-modules are parameterized by the set
$$\mathcal{C} :=\big\{(\mu, \psi_{1},\ldots,\psi_{r})\:|\:\mu\in \mathcal{C}_{r}(n)\text{ and each } \psi_{i} \text{ is an }e\text{-restricted partition of } \mu_{i}\big\}.$$
\end{corollary}

\subsection{The functors $e_{j, k}^{\lambda}$ and $f_{j, k}^{\lambda}$}
Fix a module $M$ in $\widehat{Y}_{r,n}$-$\mathbf{mod}_{\mathbb{I}}.$ From (\ref{mgammamu}), we get the following decomposition:
\begin{equation}\label{m-mu-gamma-add-add-new-add}
M=\bigoplus_{\mu\in \mathcal{C}_{r}(n); \gamma\in \mathbb{I}^{n}/\sim}M[\mu, \gamma].
\end{equation}

For each $j\in \mathbb{J},$ let $\varepsilon_{j}$ be the associated standard basis of $\mathbb{Z}^{e}.$ We denote by $\Gamma_{n}$ the set of linear combinations $\gamma=\sum_{j\in \mathbb{J}}\gamma_{j}\varepsilon_{j}$ of $\varepsilon_{j}$ such that $\gamma_{j}\in \mathbb{Z}_{\geq 0}$ for all $j$ and $\sum_{j\in \mathbb{J}}\gamma_{j}=n.$ If $\underline{s}\in \mathbb{I}^{n},$ we define its \emph{content} by $$\mathrm{cont}(\underline{s}) :=\sum_{j\in \mathbb{J}}\gamma_{j}\varepsilon_{j}\in \Gamma_{n},~~~\mathrm{where}~\gamma_{j}=\#\big\{k=1,2,\ldots,n\:|\:s_{k}=q^{j}\big\}.$$
The content map induces a canonical bijection between $\mathbb{I}^{n}/\sim$ and $\Gamma_{n},$ and we shall not distinguish between them. Then we rewrite \eqref{m-mu-gamma-add-add-new-add} as
\begin{equation}\label{m-mu-gamma}
M=\bigoplus_{\mu\in \mathcal{C}_{r}(n); \gamma\in \Gamma_{n}}M[\mu, \gamma].
\end{equation}
In fact, such a decomposition also makes sense in the category $Y_{r,n}^{\lambda}$-$\mathbf{mod}$ for all $\lambda\in \Delta.$

By Proposition \ref{cyclotomic-pbw}, it is easy to see that $Y_{r,n-1}^{\lambda}\otimes \mathbb{K}G$ is isomorphic to the subalgebra of $Y_{r,n}^{\lambda}$ generated by $X_{1},\ldots,X_{n-1},$ $t_{1},\ldots, t_{n}$ and $g_{w}$ for all $w\in \mathfrak{S}_{n-1}.$ We shall not distinguish between them.
\begin{definition}\label{ef-M}
Suppose that $M\in Y_{r,n}^{\lambda}$-$\mathbf{mod}$ and that $M=M[\mu, \gamma]$ for some $\mu\in \mathcal{C}_{r}(n)$ and $\gamma\in \Gamma_{n}.$ For each $j\in \mathbb{J}$ and $1\leq k\leq r,$ we define
$$e_{j,k}^{\lambda}M=\mathrm{Hom}_{\mathbb{K}G}\big(V_{k}, \mathrm{Res}_{Y_{r,n-1}^{\lambda}\otimes \mathbb{K}G}^{Y_{r,n}^{\lambda}}M\big)\big[\mu_{k}^{-}, \gamma-\varepsilon_{j}\big],$$
$$f_{j,k}^{\lambda}M=\Big(\mathrm{Ind}_{Y_{r,n}^{\lambda}\otimes \mathbb{K}G}^{Y_{r,n+1}^{\lambda}}\big(M\otimes V_{k}\big)\Big)\big[\mu_{k}^{+}, \gamma+\varepsilon_{j}\big].$$
\end{definition}

By \eqref{m-mu-gamma}, we can extend $e_{j, k}^{\lambda}$ (resp. $f_{j,k}^{\lambda}$) to functors from $Y_{r,n}^{\lambda}$-$\mathbf{mod}$ to $Y_{r,n-1}^{\lambda}$-$\mathbf{mod}$ (resp. from $Y_{r,n}^{\lambda}$-$\mathbf{mod}$ to $Y_{r,n+1}^{\lambda}$-$\mathbf{mod}$).

\begin{remark}\label{remark-1} When $r=1,$ the cyclotomic Yokonuma-Hecke algebra coincides with the Ariki-Koike algebra, and the functors $e_{j, k}^{\lambda}$ and $f_{j,k}^{\lambda}$ in fact coincide with the ones $e_{j}^{\lambda}$ and $f_{j}^{\lambda}$ for Ariki-Koike algebras, which are defined by Ariki and also Grojnowski; see \cite{Ari1} and \cite{Gr}.
\end{remark}

\subsection{Branching rules for $Y_{r,n}^{\lambda}$ and a crystal graph interpretation}
For a module category $\mathcal{A},$ let $K(\mathcal{A})$ be the Grothendieck group of $\mathcal{A}$ and $\mathrm{Irr}(\mathcal{A})$ be the set of pairwise non-isomorphic simple objects in $\mathcal{A}.$ For each $\lambda\in P_{+},$ we set $$K(\lambda) :=\bigoplus_{m\geq 0}K(\mathcal{H}_{m}^{\lambda}\mathbf{-mod}),\quad\text{and}\quad K(\lambda)_{\mathbb{C}} :=\mathbb{C}\otimes_{\mathbb{Z}}K(\lambda).$$

Fix one $j\in \mathbb{J}.$ Associated to the two functors $e_{j}^{\lambda}$ and $f_{j}^{\lambda}$ for Ariki-Koike algebras
in Remark \ref{remark-1}, there are two additional operators $\tilde{e}_{j}^{\lambda}$ and $\tilde{f}_{j}^{\lambda}$ on
$\coprod_{m\geq 0}\mathrm{Irr}(\mathcal{H}_{m}^{\lambda}$-$\mathbf{mod})$
by setting $\tilde{e}_{j}^{\lambda}L :=\mathrm{Soc}(e_{j}^{\lambda}L)$ and
$\tilde{f}_{j}^{\lambda}L :=\mathrm{Head}(f_{j}^{\lambda}L),$ where $L$ is a simple $\mathcal{H}_{m}^{\lambda}$-module.

Let $L(\lambda)$ be the simple highest weight $\widehat{sl}_{e}$-module of highest weight $\lambda.$ Then we have the following results.
\begin{proposition}\label{ariki-groj}{\rm (See \cite[Theorem 4.4]{Ari1} and \cite[Theorems 14.2 and 14.3]{Gr}.)}
Assume that $\lambda\in P_{+}.$ Then $K(\lambda)_{\mathbb{C}}$ is an $\widehat{sl}_{e}$-module with the Chevalley generators acting as $e_{j}^{\lambda}$ and $f_{j}^{\lambda}$ $($with $j\in \mathbb{J}),$ and is isomorphic to $L(\lambda)$ as $\widehat{sl}_{e}$-modules.

Moreover, $\coprod_{m\geq 0}\mathrm{Irr}(\mathcal{H}_{m}^{\lambda}$-$\mathbf{mod})$ is isomorphic to the crystal basis $B(\lambda)$ of the simple $\widehat{sl}_{e}$-module $L(\lambda)$ with operators $\tilde{e}_{j}^{\lambda}$ and $\tilde{f}_{j}^{\lambda}$ identified with the Kashiwara operators.
\end{proposition}


For each $\lambda\in \Delta,$ we set $$K_{T}(\lambda) :=\bigoplus_{n\geq 0}K(Y_{r,n}^{\lambda}\mathbf{-mod}),\quad\text{and}\quad K_{T}(\lambda)_{\mathbb{C}} :=\mathbb{C}\otimes_{\mathbb{Z}}K_{T}(\lambda).$$
For each $j\in \mathbb{J}$ and $1\leq k\leq r,$ we have defined two functors $e_{j,k}^{\lambda}$ and $f_{j,k}^{\lambda}$ in Definition \ref{ef-M}. They induces linear operators on $K_{T}(\lambda).$ By Theorem \ref{morita-equi}, the category equivalence $\mathcal{F}^{\lambda}$ induces a canonical linear isomorphism
\begin{equation}\label{f-lambda}
\widetilde{\mathcal{F}}^{\lambda} :K_{T}(\lambda)\overset{\sim}{\longrightarrow} K(\lambda)\otimes \cdots\otimes K(\lambda)= K(\lambda)^{\otimes r}.
\end{equation}

Observe that the functor $e_{i,k}^{\lambda}$ corresponds via $\widetilde{\mathcal{F}}^{\lambda}$ to $e_{i}^{\lambda}$ applied to the $k$-th tensor factor on the right-hand side of (\ref{f-lambda}) by Lemma \ref{branch-rule-2}. By applying Frobenius reciprocity in the context of the pair of algebras $(Y_{r,n-1}^{\lambda}\otimes \mathbb{K}G, Y_{r,n}^{\lambda})$ (see \cite[Chapter II, $\S5.1$]{Bo}) we can deduce that  $f_{i,k}^{\lambda}$ is left adjoint to $e_{i,k}^{\lambda}$ and $f_{i}^{\lambda}$ is left adjoint to $e_{i}^{\lambda};$ hence $f_{i,k}^{\lambda}$ corresponds to $f_{i}^{\lambda}$ applied to the $k$-th tensor factor on the right-hand side of (\ref{f-lambda}). Hence Theorem \ref{branch-rule-3} implies the following modular branching rule for $Y_{r,n}^{\lambda}$ under the identification of $Y_{r,n}^{\lambda}$-$\mathbf{mod}$ with a full subcategory of $\widehat{Y}_{r,n}$-$\mathbf{mod}.$ We denote by $Y_{r,(n-1,1)}^{\lambda}$ the subalgebra of $Y_{r,n}^{\lambda}$ generated by $X_{1},\ldots,X_{n},$ $t_1,\ldots,t_{n}$ and $g_{w}$ for all $w\in \mathfrak{S}_{n-1}.$

\begin{theorem}\label{cycloto-module-branching-rulesss}
Consider the simple $Y_{r,n}^{\lambda}$-module $S_{\mu}(L.)$ defined in Theorem \ref{cycloto-simple-mod-1}. Then we have
$$\mathrm{Soc}(\mathrm{Res}_{Y_{r,(n-1,1)}^{\lambda}}^{Y_{r,n}^{\lambda}}S_{\mu}(L.))\cong \bigoplus_{i\in \mathbb{J}; 1\leq k\leq r} S_{\mu_{k}^{-}}(\tilde{e}_{i}^{\lambda}L.)\otimes (V_{k}\otimes L(i)),$$
where  $L(i)$ is the one-dimensional $\mathbb{K}[X]$-module with $X$ acting as the scalar $q^{i}$ and $S_{\mu_{k}^{-}}(\tilde{e}_{i}^{\lambda}L.)$ denotes the $Y_{r,n-1}^{\lambda}$-module
$$\mathrm{Ind}_{\widehat{Y}_{r,\mu_{k}^{-}}}^{\widehat{Y}_{r,n-1}}\big(V(\mu_{k}^{-})\otimes (L_{1}\otimes\cdots\otimes \tilde{e}_{i}^{\lambda}L_{k}\otimes\cdots\otimes L_{r})\big).$$
\end{theorem}

Combining Theorem \ref{morita-equi} with Proposition \ref{ariki-groj} and Theorem \ref{cycloto-module-branching-rulesss}, we have established the following result.
\begin{theorem}
Assume that $\lambda\in \Delta.$ $K_{T}(\lambda)_{\mathbb{C}}$ affords a simple $\widehat{sl}_{e}^{\oplus r}$-module isomorphic to $L(\lambda)^{\otimes r}$ with the Chevalley generators of the $k$-th summand of $\widehat{sl}_{e}^{\oplus r}$ acting as $e_{j,k}^{\lambda}$ and $f_{j,k}^{\lambda}$ $($with $j\in \mathbb{J})$ for each $1\leq k\leq r.$

Moreover, $\coprod_{n\geq 0}\mathrm{Irr}(Y_{r,n}^{\lambda}$-$\mathbf{mod})$ $($and respectively, the modular branching graph given by Theorem \ref{cycloto-module-branching-rulesss}$)$ is isomorphic to the crystal basis $B(\lambda)^{\otimes r}$ $($and respectively, the corresponding crystal graph$)$ of the simple $\widehat{sl}_{e}^{\oplus r}$-module $L(\lambda)^{\otimes r}.$
\end{theorem}

\subsection{A block decomposition of $Y_{r,n}^{\lambda}$-$\mathbf{mod}$}
The blocks of the Ariki-Koike algebra $\mathcal{H}_{m}^{\lambda}$ over an arbitrary algebraically closed field have been classified in \cite[Theorem A]{LM}. By the Morita equivalence established in Theorem \ref{morita-equi}, the decomposition \eqref{m-mu-gamma} provides us a block decomposition of $Y_{r,n}^{\lambda}$-$\mathbf{mod}.$\\

\noindent{\bf Acknowledgements.}
The authors would like to thank Professor Weiqiang Wang for very helpful discussions. The first author was partially supported by Young Scholars Program of Shandong University and by National Natural Science Foundation of China (11601273). The second author was partially supported by National Natural Science Foundation of China (11571036).



\end{document}